\documentclass[a4paper,12pt]{article}
\usepackage[utf8]{inputenc}
\usepackage[T1]{fontenc}
\usepackage[english]{babel}
\usepackage{amssymb}
\usepackage{amsmath}
\usepackage{amsthm}
\usepackage{graphicx}
\usepackage{caption}
\usepackage[colorlinks=true,urlcolor=blue,
citecolor=red,linkcolor=blue,linktocpage,pdfpagelabels,
bookmarksnumbered,bookmarksopen]{hyperref}
\usepackage[titletoc,title]{appendix}
\numberwithin{equation}{section} 
\newcounter{cont}[section] 

\newtheorem{thm}[cont]{Theorem}
\newtheorem{prop}[cont]{Proposition}
\newtheorem{lem}[cont]{Lemma}

\theoremstyle{definition}
\newtheorem{defn}[cont]{Definition}
\newcommand{\y}{\textbf y}
\newcommand{\x}{\textbf x}

\title{\textbf{Slow motion for a hyperbolic variation of Allen-Cahn equation in one space dimension}}
\author{RAFFAELE FOLINO \\\footnotesize{Dipartimento di Ingegneria e Scienza dell'Informazione e Matematica}\\\footnotesize{Universit\`a degli Studi dell'Aquila}} 
\date{}

\begin{document}
\maketitle

\begin{abstract}
The aim of this paper is to prove that, for specific initial data $(u_0, u_1)$ and with homogeneous Neumann boundary conditions, the solution of the IBVP for a hyperbolic variation of Allen-Cahn equation on the interval $[a,b]$ shares the well-known dynamical metastability valid for the classical parabolic case. In particular, using the ``energy approach'' proposed by Bronsard and Kohn \cite{Bron-Kohn}, if $\varepsilon\ll1$ is the diffusion coefficient, we show that in a time scale of order $\varepsilon^{-k}$ nothing happens and the solution maintains the same number of transitions of its initial datum $u_0$. The novelty consists mainly in the role of the initial velocity $u_1$, which may create or eliminate transitions in later times. Numerical experiments are also provided in the particular case of the Allen-Cahn equation with relaxation.
\end{abstract}

\section{Introduction}\label{intro}
In this paper, we study the initial boundary value problem for the hyperbolic Allen-Cahn equation
\begin{equation}\label{Cauchy problem}
\begin{cases}
\tau u_{tt}+g(u)u_t=\varepsilon^2u_{xx}+f(u) \quad \qquad & x\in[a,b], \, t>0,\\
u(x,0)=u_0(x)  & x\in[a,b],\\
u_t(x,0)=u_1(x) & x\in[a,b],\\
u_x(a,t)=u_x(b,t)=0 & t>0,
\end{cases}
\end{equation} 
where $\tau$ and $\varepsilon$ are positive constants, $g$ is a strictly positive function and $f(u)=-F'(u)$, with $F$ a double well potential with wells of equal depth. We are interested in the limiting behavior of the solutions $u^\varepsilon$ as $\varepsilon\rightarrow0$ when $u_0$ has a \emph{transition layer structure} and $u_1$ is \emph{sufficiently} small (see conditions \eqref{u^eps tende a v}-\eqref{energia iniziale<}).\par
A particular case of \eqref{Cauchy problem} is the initial boundary value problem for the Allen-Cahn equation with relaxation 
\begin{equation}\label{allen-cahn-iper-g=1-tau f'}
\tau u_{tt}+(1-\tau f'(u))u_t=\varepsilon^2u_{xx}+f(u).
\end{equation}
Equation \eqref{allen-cahn-iper-g=1-tau f'} has both physical and probabilistic interpretation. Taking the limit as $\tau\rightarrow0$ in \eqref{allen-cahn-iper-g=1-tau f'}, we formally obtain the Allen-Cahn equation 
\begin{equation}\label{Allen-Cahn}
u_t=\varepsilon^2u_{xx}+f(u).
\end{equation}
From the physical point of view, this equation is the result of the classical theory of heat propagation by conduction, formulated by Joseph Fourier. It produces the same problem of the linear heat equation, that is the infinite speed of propagation. To avoid this unphysical property, in 1948 Cattaneo \cite{Cat} proposed the following modified equation for the heat flux $v$
\begin{equation}\label{Cattaneo law}
\tau v_t+v=-\varepsilon^2 u_x,
\end{equation} 
where $\tau>0$ is a relaxation time. Using the Cattaneo's law \eqref{Cattaneo law} instead of the Fourier's law, we obtain the \emph{semilinear Cattaneo system}
\begin{align}
u_t+v_x&=f(u) \label{semilinear Cattaneo1},\\
\tau v_t+\varepsilon^2u_x&=-v \label{semilinear Cattaneo2}.
\end{align} 
By differentiating equation \eqref{semilinear Cattaneo1} with respect to $t$ and differentiating \eqref{semilinear Cattaneo2} with respect to $x$, we obtain the equation \eqref{allen-cahn-iper-g=1-tau f'}. The detailed derivation of equation \eqref{allen-cahn-iper-g=1-tau f'} to describe heat propagation with finite speed can be found in Hadeler \cite{Hadeler95} and Hillen \cite{Hillen98}. Equation \eqref{allen-cahn-iper-g=1-tau f'} appears also in the description of a \emph{correlated random walk} (see Taylor \cite{Tay}, Goldstein \cite{Gol}, Kac \cite{Kac}, Zauderer \cite{Zau} and Holmes \cite{Holmes}). Existence and nonlinear stability of traveling fronts is provided in Lattanzio et al. \cite{LMPS}.\par
The asymptotic behavior of solutions of \eqref{Allen-Cahn} as $t\rightarrow \infty$ is well understood (see Matano \cite{Matano}): any solution $u(x,t)$ tends to a stationary solution as $t\rightarrow\infty$. However, in some cases, the convergence is exceedingly slow: the evolution is so slow that solutions (not stationary) \emph{appear} to be stable. This is an example of \emph{dynamical metastability}. The aim of this paper is to show that dynamical metastability is also present in the case of problem \eqref{Cauchy problem}. Before stating the main result of this article, let us do a short historical review.\par
Metastability for Allen-Cahn equation \eqref{Allen-Cahn} has been studied by several authors. There are at least two approaches to study dynamical metastability for Allen-Cahn equation. The \emph{dynamical approach} is due to Carr and Pego \cite{Carr-Pego} and Fusco and Hale \cite{Fusco-Hale}. They construct a $N$-dimensional manifold $\mathcal{M}$ consisting of functions which approximate metastable states with $N$ transition layers. If the initial datum is in a small neighborhood of $\mathcal{M}$, then the solution remains near $\mathcal{M}$ for a time proportional to $e^{C/\varepsilon}$. Based on these ideas slow motion results have been proved for several other equations and situations: for example Alikakos et al. \cite{Alikakos-et-al}, as well as Bates and Xun \cite{Bates-Xun1, Bates-Xun2}, consider the one-dimensional Cahn-Hilliard equation, Mascia and Strani \cite{Mascia-Strani} propose a general framework suited for parabolic equations in one dimensional bounded domains and apply such approach to scalar viscous conservation laws. \par
On the other hand, Bronsard and Kohn \cite{Bron-Kohn} proposed a different approach, based on the underlying variational structure of the equation and known as \emph{energy approach}. Using standard energy estimates, they showed that if the initial datum $u_0$ has a ``transition layer structure'', i.e. $u_0\approx\pm1$ except near finitely many transition points, then the solution maintains this structure and the transition points move slower than any power of $\varepsilon$. Grant \cite{Grant}, imposing stronger conditions on the initial datum, improved their method to obtain the exponential upper bound $O(e^{-C/\varepsilon})$ for the speed of the transition points. The energy approach has been applied by Bronsard and Hilhorst \cite{Bron-Hilh} to the one-dimensional Cahn-Hilliard equation, by Grant \cite{Grant} to prove exponentially slow motion for Cahn-Morral systems, by Kalies et al. \cite{Kal-VdV-Wan} to study slow motion of high-order systems and by Otto and Reznikoff \cite{Otto-Rez} to general gradient flows. \par
Each of these methods has its advantages and drawbacks. The dynamical approach gives the exact order of the speed of slow motion, but the proofs are complicated and lenghty. The energy approach is fairly simple and provides a rather clear and intuitive explanation for slow motion, but it gives only an upper bound for the speed.\par
In this paper, we show that the energy approach of Bronsard and Kohn can be adapted for the hyperbolic Allen-Cahn equation to obtain the above sketched results. The key point hinges on the use of the modified energy functional
$$E_\varepsilon[u^\varepsilon,u_t^\varepsilon](t):=\int_a^b\left[\frac\tau{2\varepsilon} u^\varepsilon_t(x,t)^2+\frac\varepsilon2 u^\varepsilon_x(x,t)^2+\frac{F(u^\varepsilon(x,t))}\varepsilon\right]dx,$$
with $-F'(u)=f(u)$, for which the equality
\begin{equation}\label{energy variation}
\varepsilon^{-1}\int_0^T\int_a^b g(u^\varepsilon)(u^\varepsilon_t)^2 dxdt=E_\varepsilon[u^\varepsilon,u^\varepsilon_t](0)-E_\varepsilon[u^\varepsilon,u^\varepsilon_t](T)
\end{equation}
holds. The proof of \eqref{energy variation} is in Appendix \ref{exist-uniq} (see Proposition \ref{prop-var-ener}). If $g(u)\geq \sigma>0$ for any $u\in\mathbb{R}$, it follows that
\begin{equation}\label{energy dissipation}
E_\varepsilon[u^\varepsilon,u^\varepsilon_t](0)-E_\varepsilon[u^\varepsilon,u^\varepsilon_t](T)\geq \sigma\varepsilon^{-1}\int_0^T\int_a^b{(u^\varepsilon_t)}^2dxdt
\end{equation}
and so $E_\varepsilon$ is a decreasing function on $t$ along the solutions of \eqref{Cauchy problem}. Thanks to inequality \eqref{energy dissipation}, we can use the energy approach of Bronsard and Kohn.\par
We suppose that $g$ is locally Lipschitz continuous on $\mathbb{R}$ and there exists a constant $\sigma>0$ such that
\begin{equation}\label{g strettamente positiva}
g(s)\geq \sigma, \qquad \qquad \forall\, s\in\mathbb{R}. 
\end{equation}
Regarding $f$, we suppose that $f(u)=-F'(u)$ with $F$ satisfying
\begin{equation}\label{hp-F-1}
F\in C^3(\mathbb{R}) \quad \mbox{and} \quad F(u)\rightarrow+\infty \;\mbox{ as } \;|u|\rightarrow\infty;
\end{equation}
\begin{equation}\label{hp-F-2}
F(\pm1)=F'(\pm1)=0, \quad F''(-1)>0 \quad \mbox{and} \quad F''(1)>0.
\end{equation}
Finally, we suppose that
\begin{equation}\label{hp-F-3}
\left\{u\in\mathbb{R}: F'(u)=0, F(u)\leq 0\right\}=\left\{-1,1\right\}.
\end{equation}
In other words, $F$ is a smooth, nonnegative function with global minimum equal to $0$ reached only at $u=-1$ and $u=1$. We call $F$ a \emph{bistable} potential with both wells of equal depth, because both $u=-1$ and $u=1$ are stable equilibria of the ordinary differential equation $u_t=-F'(u)$ and $F(-1)=F(1)$. The simplest example of a function $F$ satisfying assumptions \eqref{hp-F-1}, \eqref{hp-F-2}, \eqref{hp-F-3} is $F(u)=\frac14(u^2-1)^2$. Note however that $F$ is allowed to have positive critical values, including local minima. \par   
For the assumptions \eqref{g strettamente positiva}-\eqref{hp-F-3}, there exist travelling wave solutions of the hyperbolic Allen-Cahn equation, $u(x,t)=\phi(x-ct)$, connecting the stable equilibria $u=-1$ and $u=1$ if and only if $c=0$. Indeed, since $g$ is strictly positive, the sign of $c$ depends on the sign of the integral $\int_{-1}^1 f(s)ds$ that is equal to $0$ for \eqref{hp-F-2}. It follows that the profiles $\phi$ are independent of time and are solutions of the boundary value problem
$$\begin{cases}
\varepsilon^2 \phi''(y)+f(\phi(y))=0\,, \qquad y\in\mathbb{R},\\
\phi(-\infty)=-1\,, \quad \phi(+\infty)=1,
\end{cases}$$
that is the same of the classic Allen-Cahn equation \eqref{Allen-Cahn}.\par
The existence of travelling wave solutions with speed $c=0$, connecting the stable equilibria $u=-1$ and $u=1$, suggests that, if $u_0$ has a ``transition layer structure'' and $u_1=0$, layer migrations will be slow. Using energy methods introduced by \mbox{Bronsard and Kohn \cite{Bron-Kohn}} to study metastability of solutions of ${u_t=\varepsilon^2u_{xx}+u-u^3}$, we show that, under certain hypotheses on initial data $u_0$ and $u_1$, the solution of the initial boundary value problem \eqref{Cauchy problem} maintains its transition layer structure and the transition points move slower than any power of $\varepsilon$. \par
Let $a<y_1<y_2<b$, we have that
$$\int_{y_1}^{y_2}\left[\frac\varepsilon2 u_x^2+\frac{F(u)}\varepsilon\right]dx\geq\sqrt2\int_{y_1}^{y_2} u_x\sqrt{F(u)}dx= \sqrt2\int_{u(y_1)}^{u(y_2)}\sqrt{F(s)}ds.$$
If we assume that $u(y_1)=-1$ and $u(y_2)=1$, we see that the energy of a transition between $-1$ and $+1$ is greater than or equal to
\begin{equation}\label{c0}
c_0:=\sqrt{2}\int_{-1}^1 \sqrt{F(s)}ds.
\end{equation} 
We fix for the remainder of this section an integer $N\geq1$ and a piecewise constant function $v:(a,b)\rightarrow\{-1,+1\}$ with exactly $N$ discontinuities. We assume that the initial data $u_0$, $u_1$ depend on $\varepsilon$ and 
\begin{equation}\label{u^eps tende a v}
\lim_{\varepsilon\rightarrow 0} \|u_0^\varepsilon-v\|_{L^1}=0.
\end{equation}
In addition, we suppose that there exists $C>0$ such that for all sufficiently \mbox{small $\varepsilon$}
\begin{equation}\label{energia iniziale<}
E_\varepsilon[u^\varepsilon_0, u_1^\varepsilon]\leq Nc_0+C\varepsilon^k,
\end{equation}
with $c_0$ as in \eqref{c0} and $k$ a positive integer. \par
The condition \eqref{u^eps tende a v} fixes the number of transitions between $-1$ and $+1$ in the initial profile and their relative positions as $\varepsilon\rightarrow0$. If $u^\varepsilon_0$ makes these transitions then 
\begin{equation}\label{minimum energy N transitions}
\int_a^b\left[\frac\varepsilon2{(u^\varepsilon_0)}_x^2+\frac{F(u^\varepsilon_0)}\varepsilon\right]dx\geq Nc_0.
\end{equation}
More precisely, it can be shown (see Modica \cite{Modica}, Sternberg \cite{Sternberg}) that if $\{v^\varepsilon\}$ is a sequence that converges to $v$ in $L^1(a,b)$, then
$$\liminf_{\varepsilon\rightarrow0}\int_a^b\left[\frac\varepsilon2{(v^\varepsilon_x)}^2+\frac{F(v^\varepsilon)}\varepsilon\right]dx \geq Nc_0\,,$$
with equality if the sequence $\{v^\varepsilon\}$ is chosen properly. Therefore, the condition \eqref{energia iniziale<} demands that the energy at the time $t=0$ exceeds at most $C\varepsilon^k$ the minimum possible to have $N$ transitions. In particular, for \eqref{minimum energy N transitions} we have that
\begin{equation}\label{velocita iniziale piccola}
\int_a^b \frac{\tau}{2\varepsilon}{u^\varepsilon_1(x)}^2dx\leq C\varepsilon^k,
\end{equation}
\begin{equation}\label{hp-energia(u0)}
\int_a^b\left[\frac\varepsilon2{(u^\varepsilon_0)}_x^2+\frac{F(u^\varepsilon_0)}\varepsilon\right]dx \leq Nc_0+C\varepsilon^k\,.
\end{equation}
\begin{defn}
If the family $u^\varepsilon$ satisfies the conditions \eqref{u^eps tende a v} and \eqref{hp-energia(u0)}, we say that $u^\varepsilon$ has a \emph{transition layer structure of order $k$}.
\end{defn}
In other words, the hypotheses on initial data can be described as follows: the initial profile has a transition layer structure of order $k$ and the $L^2$-norm of the initial velocity is less than or equal to $C\varepsilon^\frac{k+1}2$. Under these hypotheses, we can say that nothing happens on a time scale of order $\varepsilon^{-k}$. This is the concern of our main result.
\begin{thm}\label{metastability-thm}
We consider the initial boundary value problem \eqref{Cauchy problem} with $F$ and $g$ satisfying \eqref{g strettamente positiva}-\eqref{hp-F-3}. We suppose that the initial data $u^\varepsilon_0, u^\varepsilon_1$ satisfy \eqref{u^eps tende a v} and \eqref{energia iniziale<} for some $k>0$. Then for any $m>0$
\begin{equation}\label{metastability-limit}
\sup_{0\leq\,t\,\leq m\varepsilon^{-k}}\|u^\varepsilon(\cdot,t)-v\|_{L^1}\xrightarrow[\varepsilon\rightarrow0]{}0.
\end{equation}
\end{thm}
Note that the statement of Theorem \ref{metastability-thm} and the hypotheses on $u^\varepsilon_0$ are the same of the parabolic case (cfr. \cite[Theorem $4.1$]{Bron-Kohn}). The proof of Theorem \ref{metastability-thm} is in Section \ref{metastability}.\par 
We observe that the result \eqref{metastability-limit} holds for any strictly positive function $g$. In the relaxation case \eqref{allen-cahn-iper-g=1-tau f'}, we can state a more precise result on the asymptotic behavior of solutions as $t\rightarrow\infty$. Hillen \cite{Hillen} studied asymptotic behavior of solutions of system \eqref{semilinear Cattaneo1}-\eqref{semilinear Cattaneo2} in the interval $[a,b]$ with homogeneous Dirichlet or Neumann boundary conditions. Under appropriate assumptions on $f$, Hillen proved global existence of solutions and showed that each solution converges to a stationary solution as $t\rightarrow\infty$. If $u^\varepsilon_0$, $u^\varepsilon_1$ satisfy appropriate conditions and $\varepsilon$ is small, the solution $u^\varepsilon(x,t)$ tends to a stationary solution as $t\rightarrow\infty$, but the convergence is very slow: in a time scale of order $\varepsilon^{-k}$ the solution has a transition layer structure of order $k$. \par
The rest of this paper is organized as follows. Section \ref{metastability} contains the proof of the Theorem \ref{metastability-thm} and the result on the velocity of the transition points (cfr. Theorem \ref{Teorema-distanza interfacce}). In Section \ref{numerical example} we present some numerical examples in the relaxation case, i.e. $g(u)=1-\tau f'(u)$, to show slow motion of solutions of \eqref{allen-cahn-iper-g=1-tau f'} with $f(u)=u-u^3$. If the initial velocity $u_1$ is equal to $0$, there are no important differences with the parabolic case \eqref{Allen-Cahn}; it is interesting to study examples where $u_1$ is not constantly equal to $0$. We will exhibit an example where the initial profile $u_0$ is constant and thanks to the initial velocity $u_1$ a metastable state is formed. At the same way, starting from a profile $u_0$ having a transition layer structure, a transition can be eliminated choosing opportunely the initial velocity $u_1$. This is a fundamental difference with the Allen-Cahn equation \eqref{Allen-Cahn}, where the metastable state has $N$ transitions if and only if $u_0$ has $N$ transitions. Finally, in Appendix \ref{exist-uniq} we prove that the initial boundary value problem \eqref{Cauchy problem} is globally well-posed for positive times in the space $H^1([a,b])\times L^2(a,b)$. Using a classical semigroup argument, we show that there exists a unique \emph{mild solution} for any initial data $(u_0,u_1)\in H^1\times L^2$ and this solution satisfies the equality \eqref{energy variation} for any $T>0$. \par
In this paper, the attention is focused on the energy approach of Bronsard and Kohn for the scalar case of the hyperbolic Allen-Cahn equation. This approach can be refined to obtain persistence of the layered structure for time intervals of $O(e^{C/\varepsilon})$ and extended to the case when $u$ is vector-valued. These topics, as well as the study of metastability for hyperbolic Allen-Cahn equation with the dynamical approach of Carr and Pego, are addressed in \cite{Folino,FLM}. 

\section{Slow evolution}\label{metastability}
In this section we study metastability of solutions of the equation
\begin{equation}\label{allen-cahn-iper}
\tau u_{tt}+g(u)u_t=\varepsilon^2u_{xx}+f(u), \qquad \quad (x,t)\in [a,b]\times(0,T),
\end{equation}
with homogeneous Neumann boundary conditions
\begin{equation}\label{Neumann-bordo}
u_x(a,t)=u_x(b,t)=0 \qquad\quad \forall t>0
\end{equation}
and initial data 
\begin{equation}\label{cond-iniz}
u(x,0)=u_0(x), \qquad u_t(x,0)=u_1(x), \qquad \qquad x\in[a,b].
\end{equation}
The first step to prove Theorem \ref{metastability-thm} is to show a lower bound on the energy. The result is purely variational in character and it has been proved by Bronsard and Kohn \cite{Bron-Kohn} in the case $F(u)=\frac{u^2}2-\frac{u^4}4$. We extend the result to the case of a function $F$ that satisfies \eqref{hp-F-1}, \eqref{hp-F-2} and \eqref{hp-F-3}, adapting the approach in \cite{Bron-Kohn}.
\begin{prop}\label{stima dal basso energia}
Let $F:\mathbb{R}\rightarrow\mathbb{R}$ be a function satisfying \eqref{hp-F-1}, \eqref{hp-F-2}, \eqref{hp-F-3}, $v:(a,b)\rightarrow\{-1,+1\}$ a piecewise constant function with exactly $N$ discontinuities and let $l$ be a positive integer. Then there exist constants $\delta_l>0$ and $c_l>0$ such that if $w\in H^1$ satisfies
\begin{equation}\label{|w-v|_L^1<delta_l}
\|w-v\|_{L^1}\leq\delta_l
\end{equation}
and
\begin{equation}\label{E_p(w)<Nc0+eps^l}
\int_a^b\left[\frac\varepsilon2 w^2_x+\frac{F(w)}\varepsilon\right]dx\leq Nc_0+\varepsilon^l
\end{equation}
with $\varepsilon$ sufficiently small, then
\begin{equation}\label{E_p(w)>Nc0-eps^l}
\int_a^b\left[\frac\varepsilon2 w^2_x+\frac{F(w)}\varepsilon\right]dx\geq Nc_0-c_l\varepsilon^l.
\end{equation}
\end{prop}
\begin{proof}
Firstly, we consider the case $N=1$. Let $\gamma$ be the point of discontinuity of $v$. We may assume that $v=-1$ on $(a,\gamma)$ (if necessary, replace $v$ and $w$ by $-v$ and $-w$). We choose $\delta_l$ sufficiently small so that
$$(\gamma-2l\delta_l,\gamma+2l\delta_l)\subset(a,b).$$
We show that from hypotheses \eqref{|w-v|_L^1<delta_l} and \eqref{E_p(w)<Nc0+eps^l} it follows that there exist two points $x_1{\in(\gamma-2\delta_l,\gamma)}$ and $y_1\in(\gamma,\gamma+2\delta_l)$ such that
\begin{equation}\label{x_1,y_1}
w(x_1)\leq-1+C\varepsilon^\frac12, \qquad w(y_1)\geq1-C\varepsilon^\frac12\,.
\end{equation}
Here and throughout, $C$ represents a positive constant that is independent of $\varepsilon$, whose value may change from line to line. From hypothesis \eqref{|w-v|_L^1<delta_l} we have
\begin{equation}\label{int su (gamma,b)}
\int_\gamma^b|w-1|\leq\delta_l.
\end{equation}
If we denote by $S^-:=\{y:w(y)\leq0\}$ and by $S^+:=\{y:w(y)>0\}$, then it follows from \eqref{int su (gamma,b)} that $\mbox{meas}(S^-\cap(\gamma,b))\leq\delta_l$ and hence that
$$\mbox{meas}(S^+\cap(\gamma,\gamma+2\delta_l))\geq\delta_l.$$
Now, from \eqref{E_p(w)<Nc0+eps^l}, we obtain
$$\int_{S^+\cap(\gamma,\gamma+2\delta_l)}\frac{F(w)}\varepsilon\, dx\leq c_0+1,$$
and therefore there exists $y_1\in S^+\cap(\gamma,\gamma+2\delta_l)$ such that
\begin{equation}\label{F(w(y1))<Ceps}
F(w(y_1))\leq C\varepsilon, \qquad \quad C=\frac{c_0+1}{\delta_l}.
\end{equation}
Since $w(y_1)>0$, if $\varepsilon$ is sufficiently small, it follows from \eqref{F(w(y1))<Ceps} that
\begin{equation}\label{solo y1}
w(y_1)\geq1-C_1\varepsilon^\frac12.
\end{equation}
Indeed, let $\alpha_0=F''(1)>0$ and we choose $\beta_0>0$ small enough so that
\begin{equation}\label{F''(w)}
\frac{\alpha_0}2\leq F''(w)\leq 2\alpha_0, \qquad \quad \forall\,w\in[1-\beta_0,1+\beta_0].
\end{equation}
Integrating \eqref{F''(w)} and using \eqref{hp-F-2} we have
\begin{equation}\label{F(w) compresa tra due parabole}
\frac{\alpha_0}4{(w-1)}^2\leq F(w)\leq\alpha_0{(w-1)}^2,
\end{equation}
for all $w\in[1-\beta_0,1+\beta_0]$. If $\varepsilon$ is sufficiently small so that \eqref{F(w(y1))<Ceps} implies $w(y_1)\in[1-\beta_0,1+\beta_0]$, then \eqref{solo y1} follows from \eqref{F(w(y1))<Ceps} and \eqref{F(w) compresa tra due parabole}. The existence of $x_1{\in S^-\cap(\gamma-2\delta_l,\gamma)}$ such that $w(x_1)\leq-1+C\varepsilon^\frac12$ is proved similarly.\par
Now, let us prove that from \eqref{x_1,y_1} follows \eqref{E_p(w)>Nc0-eps^l} with $l=1$. We have
\begin{equation}\label{stima-l=1}
\int_{x_1}^{y_1}\left[\frac\varepsilon2 w_x^2+\frac{F(w)}\varepsilon\right]dx\geq\sqrt2\int_{x_1}^{y_1} |w_x|\sqrt{F(w)}dx = \int_{x_1}^{y_1}\left|\frac{d}{dx}\Psi(w)\right|dx,
\end{equation}
where $\Psi'(x)=\sqrt{2F(x)}$. Note that $c_0$ in \eqref{c0} coincides with $\Psi(1)-\Psi(-1)$. Using the monotonicity of $\Psi$ and inequalities \eqref{F(w) compresa tra due parabole}, we conclude that
\begin{align}
\int_{x_1}^{y_1}\left|\frac{d}{dx}\Psi(w)\right|dx &\geq \Psi(w(y_1))-\Psi(w(x_1))\nonumber\\
&\geq c_0-\sqrt2\int_{1-C\varepsilon^\frac12}^1 \sqrt{F(s)} ds -\sqrt2\int_{-1}^{-1+C\varepsilon^\frac12}\sqrt{F(s)}ds \nonumber\\ &\geq c_0-c_1\varepsilon. \label{stima-caso-l=1 (2)}
\end{align}
This implies \eqref{E_p(w)>Nc0-eps^l} when $l=1$. \par
We argue inductively to prove the following assertions for $1\leq k\leq l$: there exist $x_k\in(\gamma-2k\delta_l,\gamma)$ and $y_k\in(\gamma,\gamma+2k\delta_l)$ such that
\begin{equation}\label{x_k,y_k}
w(x_k)\leq-1+C\varepsilon^\frac{k}2, \qquad w(y_k)\geq 1-C\varepsilon^\frac{k}2,
\end{equation}
with $C=C(k)$ independent of $\varepsilon$ and
\begin{equation}\label{E_p(w)-xk,yk}
\int_{x_k}^{y_k}\left[\frac\varepsilon2 w^2_x+\frac{F(w)}\varepsilon\right]dx\geq c_0-c_k\varepsilon^k.
\end{equation}
We have already completed the initial step, $k=1$. Let us show that for $k<l$, from \eqref{E_p(w)-xk,yk} it follows that there exist $x_{k+1}\in(x_k-2\delta_l,x_k)$ and $y_{k+1}\in(y_k,y_k+2\delta_l)$ such that
\begin{equation}\label{x_k+1,y_k+1}
w(x_{k+1})\leq-1+C\varepsilon^\frac{k+1}2, \qquad w(y_{k+1})\geq 1-C\varepsilon^\frac{k+1}2,
\end{equation}   
By \eqref{|w-v|_L^1<delta_l} we have
\begin{equation}\label{meas>delta_l-k}
\mbox{meas}(S^+\cap(y_k,y_k+2\delta_l))\geq\delta_l.
\end{equation}
Furthermore, from \eqref{E_p(w)-xk,yk} and \eqref{E_p(w)<Nc0+eps^l} it follows that
$$\int_{y_k}^b\frac{F(w)}\varepsilon dx\leq C\varepsilon^k,$$
and hence that
\begin{equation}\label{int F(w)<Ceps^k+1}
\int_{S^+\cap(y_k,y_k+2\delta_l)} F(w)\,dx\leq C\varepsilon^{k+1}.
\end{equation}
Thanks to \eqref{meas>delta_l-k} and \eqref{int F(w)<Ceps^k+1}, we can say that there exists $y_{k+1}\in S^+\cap(y_k,y_k+2\delta_l)$ such that
\begin{equation}\label{F(w(y_k+1))<Ceps^k+1}
F(w(y_{k+1}))\leq \frac{C}{\delta_l}\varepsilon^{k+1}.
\end{equation}
Arguing as for \eqref{solo y1}, using \eqref{F(w(y_k+1))<Ceps^k+1} and \eqref{F(w) compresa tra due parabole} we conclude the existence of $y_{k+1}\in(y_k,y_k+2\delta_l)$ with the desired property. Similarly, it can be proved the existence of $x_{k+1}\in(x_k-2\delta_l,x_k)$. To complete the induction argument we must show that if $k<l$, then \eqref{x_k+1,y_k+1} implies
$$\int_{x_{k+1}}^{y_{k+1}}\left[\frac\varepsilon2 w^2_x+\frac{F(w)}\varepsilon\right]dx\geq c_0-c_{k+1}\varepsilon^{k+1}.$$
 Arguing as \eqref{stima-l=1} and using \eqref{F(w) compresa tra due parabole}-\eqref{x_k+1,y_k+1} we have
\begin{align*}
\int_{x_{k+1}}^{y_{k+1}}\left[\frac\varepsilon2 w_x^2+\frac{F(w)}\varepsilon\right]dx &\geq \Psi(y_{k+1})-\Psi(x_{k+1}) \\
&\geq c_0\!-\!\sqrt2\left(\int_{1-C\varepsilon^\frac{k+1}2}^1\! \sqrt{F(s)} ds +\int_{-1}^{-1+C\varepsilon^\frac{k+1}2}\!\sqrt{F(s)}ds\right)\\
&\geq c_0-c_{k+1}\varepsilon^{k+1}.
\end{align*}
Since \eqref{E_p(w)-xk,yk} with $k=l$ implies \eqref{E_p(w)>Nc0-eps^l}, we have completed the proof in case $N=1$.\par
The preceding argument is fundamentally local in character, so it is readily adapted to the case $N>1$. Let $v$ have $N$ discontinuities at points $\gamma_1<\gamma_2<\ldots<\gamma_N$; for ease of notation we set $a=\gamma_0, \gamma_{N+1}=b$. We argue as in the case $N=1$ in each point of discontinuity $\gamma_i$. The constant $\delta_l$ must be sufficiently small so that 
$$\gamma_i+2l\delta_l<\gamma_{i+1}-2l\delta_l \qquad \quad 0\leq i\leq N.$$
We may assume without loss of generality that $v=-1$ on $(a,\gamma_i)$. Arguing as for \eqref{x_1,y_1} we obtain the existence of $x^i_1\in(\gamma_i-2\delta_l,\gamma_i)$ and $y^i_1\in(\gamma_i,\gamma_i+2\delta_l)$ such that
$$w(x^i_1)\approx (-1)^i,\qquad \quad w(y^i_1)\approx(-1)^{i+1},$$
$$F(w(x^i_1))\leq C\varepsilon, \qquad \quad F(w(y^i_1))\leq C\varepsilon.$$
On each interval $(x_1^i,y_1^i)$, we can estimate as in \eqref{stima-l=1}-\eqref{stima-caso-l=1 (2)}; adding these estimate gives
$$\sum_{i=1}^N\int_{x_1^i}^{y_1^i}\left[\frac\varepsilon2 w_x^2+\frac{F(w)}\varepsilon\right]dx\geq Nc_0-c_1\varepsilon$$
and so we have \eqref{E_p(w)>Nc0-eps^l} with $l=1$. Arguing inductively as done in case $N=1$, we obtain \eqref{E_p(w)>Nc0-eps^l} for the general case $l\geq2$.
\end{proof}
Proposition \ref{stima dal basso energia} is fundamental to prove Theorem \ref{metastability-thm}.
\begin{prop}\label{norma_L2<Ceps^k+1}
Assume that $F$ and $g$ satisfy \eqref{g strettamente positiva}-\eqref{hp-F-3} and that the initial data $u^\varepsilon_0$, $u^\varepsilon_1$ satisfy \eqref{u^eps tende a v} and \eqref{energia iniziale<} for some $k>0$. Let $u^\varepsilon$ be the solution of \eqref{allen-cahn-iper}-\eqref{Neumann-bordo}-\eqref{cond-iniz}. Then there exist positive constants $C_1,C_2$ (depending on $v$, $g$ and $k$, but not on $\varepsilon$) such that
\begin{equation}\label{norma-L2(u_t)<eps^k+1}
\int_0^{C_1\varepsilon^{-(k+1)}}\!\!{\|u^\varepsilon_t\|}^2_{L^2}\,dt\leq C_2\varepsilon^{k+1}
\end{equation}
for all sufficiently small $\varepsilon$.
\end{prop}
\begin{proof}
By \eqref{u^eps tende a v}, we can say that for all sufficiently small $\varepsilon$ 
\begin{equation}\label{int|u^eps-v|<delta}
\|u^\varepsilon_0-v\|_{L^1}\leq\frac12\delta_k,
\end{equation}
with $\delta_k$ as in Proposition \ref{stima dal basso energia}. In the following, we consider only values of $\varepsilon$ for which \eqref{energia iniziale<} and \eqref{int|u^eps-v|<delta} hold.\par
Using Proposition \ref{stima dal basso energia}, we show that if $T=T(\varepsilon)>0$ satisfies
\begin{equation}\label{norma_L1-u_t^eps}
\int_0^{T(\varepsilon)}{\|u_t^\varepsilon\|}_{L^1}\,dt\leq\frac12\delta_k,
\end{equation}
then
\begin{equation}\label{E(T(eps))>Nc0-c_keps^k}
E_\varepsilon[u^\varepsilon,u^\varepsilon_t](T_\varepsilon)\geq Nc_0-c_k\varepsilon^k, 
\end{equation}
for some $c_k>0$. Let $w(x)=u^\varepsilon(x,T(\varepsilon))$. By \eqref{energy dissipation} it follows that $E_\varepsilon[u^\varepsilon,u^\varepsilon_t]$ decreases in time and hence, thanks to hypothesis \eqref{energia iniziale<}, we have 
$$E_\varepsilon[u^\varepsilon,u^\varepsilon_t](T_\varepsilon)\leq E_\varepsilon[u^\varepsilon,u^\varepsilon_t](0)\leq Nc_0+C\varepsilon^k.$$  
Therefore, $w$ satisfies condition \eqref{E_p(w)<Nc0+eps^l}. Furthermore, if \eqref{int|u^eps-v|<delta} and \eqref{norma_L1-u_t^eps} hold, then
$${\|v-w\|}_{L^1}\leq{\|v-u^\varepsilon_0\|}_{L^1}+{\|u^\varepsilon_0-w\|}_{L^1}\leq \frac12\delta_k+\frac12\delta_k=\delta_k.$$
Thus $w$ satisfies condition \eqref{|w-v|_L^1<delta_l} and, applying Proposition \ref{stima dal basso energia}, we obtain \eqref{E(T(eps))>Nc0-c_keps^k}.\par
Substitution of \eqref{energia iniziale<} and \eqref{E(T(eps))>Nc0-c_keps^k} in \eqref{energy dissipation} yields
\begin{equation}\label{risultato con T(eps)}
\int_0^{T(\varepsilon)} {\|u^\varepsilon_t\|}^2_{L^2}\,dt\leq C_2\varepsilon^{k+1},
\end{equation}
where $C_2=(C+c_k)/\sigma$. Hence to prove \eqref{norma-L2(u_t)<eps^k+1} we must simply show that \eqref{norma_L1-u_t^eps} holds with $T(\varepsilon)\geq C_1\varepsilon^{-(k+1)}$. If
$$\int_0^{+\infty}{\|u_t^\varepsilon\|}_{L^1}\,dt\leq\frac12\delta_k,$$
then there is nothing to prove; otherwise choose $T_1(\varepsilon)$ such that
$$\int_0^{T_1(\varepsilon)}{\|u_t^\varepsilon\|}_{L^1}\,dt=\frac12\delta_k.$$
Using Cauchy-Schwarz inequality and \eqref{risultato con T(eps)}, we obtain
\begin{align*}
\frac12\delta_k &\leq {((b-a)T_1(\varepsilon))}^{\frac12}{\left(\int_0^{T_1(\varepsilon)}{\|u_t^\varepsilon\|}_{L^2}^2\,dt\right)}^{\frac12}\\
&\leq {(C_2(b-a)T_1(\varepsilon))}^{\frac12}\varepsilon^{\frac{k+1}2}\,,
\end{align*}
so 
$$T_1(\varepsilon)\geq C_1\varepsilon^{-(k+1)}$$
and the proof is complete.
\end{proof}
Now, we can prove Theorem \ref{metastability-thm}.
\begin{proof}[Proof of Theorem \ref{metastability-thm}]
Let $m>0$. Triangle inequality gives:
\begin{equation}\label{dis-triangle}
\|u^\varepsilon(\cdot,t)-v\|_{L^1}\leq\|u^\varepsilon(\cdot,t)-u^\varepsilon_0\|_{L^1}+\|u^\varepsilon_0-v\|_{L^1},
\end{equation}
for all $t\in[0,m\varepsilon^{-k}]$. The last term of the right hand side of \eqref{dis-triangle} tends to zero as $\varepsilon\rightarrow0$ thanks to hypothesis \eqref{u^eps tende a v}. On the other hand,
\begin{equation}\label{u^eps(t)-u^eps(0)}
\sup_{0\leq\,t\,\leq m\varepsilon^{-k}}{\|u^\varepsilon(\cdot,t)-u^\varepsilon_0\|}_{L^1} \leq\int_0^{m\varepsilon^{-k}}{\|u_t^\varepsilon\|}_{L^1}\,dt.
\end{equation} 
To estimate the last term of \eqref{u^eps(t)-u^eps(0)}, we use Proposition \ref{norma_L2<Ceps^k+1}; if $\varepsilon$ is small enough so that $C_1\varepsilon^{-1}\geq m$, using \eqref{norma-L2(u_t)<eps^k+1} and Cauchy-Schwarz inequality, we have  
$$\int_0^{m\varepsilon^{-k}}{\|u_t^\varepsilon\|}_{L^1}\,dt\leq m^\frac12\varepsilon^{-\frac k2}(b-a)^\frac12 ({C_2\varepsilon^{k+1})}^\frac12\leq {(m(b-a)C_2)}^\frac12\varepsilon^\frac12,$$
and so
\begin{equation}\label{norma-L1-u_t^eps}
\lim_{\varepsilon\rightarrow0}\int_0^{m\varepsilon^{-k}}{\|u_t^\varepsilon\|}_{L^1}\,dt=0.
\end{equation}
By combining \eqref{u^eps tende a v}, \eqref{dis-triangle}, \eqref{u^eps(t)-u^eps(0)} and \eqref{norma-L1-u_t^eps}, we obtain \eqref{metastability-limit}.
\end{proof}
Now, let us study the motion of the transition points. To do this, we present a preliminary Lemma concerning the structure of $u^\varepsilon(\cdot,t)$. Like Proposition \ref{stima dal basso energia}, this Lemma is purely variational in character: equation \eqref{allen-cahn-iper} plays no role.
\begin{lem}\label{lemma-struttura-u^eps}
Let $F:\mathbb{R}\rightarrow\mathbb{R}$ be a function satisfying \eqref{hp-F-1}, \eqref{hp-F-2}, \eqref{hp-F-3} and let $v:(a,b)\rightarrow\{-1,+1\}$ be a piecewise constant function with $N$ discontinuities at points $\gamma_1<\gamma_2<\cdots<\gamma_N$. There exists a constant $\bar\delta>0$ such that for any $\delta<\bar\delta$ holds the following property: if $w^\varepsilon\in H^1$ satisfies
$$\|w^\varepsilon-v\|_{L^1}\leq\frac\delta2, \qquad \quad \mbox{ and } \qquad \quad E_\varepsilon\left[w^\varepsilon,0\right]\leq Nc_0+C\varepsilon$$
for sufficiently small $\varepsilon$, then there exist intervals $(x_i,y_i)$ containing $\gamma_i$ such that 
\begin{align}
|x_i-\gamma_i|\leq\delta, &\qquad \qquad |y_i-\gamma_i|\leq\delta \label{x_i,y_i,gamma_i}\\
\lim_{\varepsilon\rightarrow0}F(w^\varepsilon(x))=0 \qquad\quad&\quad\mbox{ for } \qquad\quad x\not\in\bigcup_{i=1}^N (x_i,y_i). \label{F(w(x)) tends 0}
\end{align}
\end{lem}
\begin{proof}
For ease of notation, we set $\gamma_0=a, \gamma_{N+1}=b$. Let us define
\begin{equation}\label{delta-bar}
\bar\delta:=\min_{0\leq i\leq N}\frac{\gamma_{i+1}-\gamma_i}2.
\end{equation}
Fix $0<\delta<\bar\delta$ so that
$$\gamma_i+\delta<\gamma_{i+1}-\delta \qquad \quad 0\leq i\leq N.$$
Arguing as for \eqref{x_1,y_1} and \eqref{stima-l=1}-\eqref{stima-caso-l=1 (2)} we obtain the existence of points $x_i,y_i$ satisfying \eqref{x_i,y_i,gamma_i} and
\begin{equation}\label{F(w(x_i))<Ceps/delta}
F(w^\varepsilon(x_i))\leq \frac{C\varepsilon}\delta, \qquad \quad F(w^\varepsilon(y_i))\leq\frac{C\varepsilon}\delta,
\end{equation}
\begin{equation}\label{E_p(w)>c0-Ceps/delta}
\int_{x_i}^{y_i}\left[\frac\varepsilon2 {(w^\varepsilon_x)}^2+\frac{F(w^\varepsilon)}\varepsilon\right]dx\geq c_0-\frac{C\varepsilon}\delta.
\end{equation}
Let $V=[a,b]\backslash\bigcup_{i=1}^N(x_i,y_i)$. Then by \eqref{E_p(w)>c0-Ceps/delta}
$$\int_V\left[\frac\varepsilon2 {(w^\varepsilon_x)}^2+\frac{F(w^\varepsilon)}\varepsilon\right]dx\leq \frac{C\varepsilon}\delta.$$
Arguing as in \eqref{stima-l=1} we conclude that
$$\int_V\left|\frac{d}{dx}\Psi(w^\varepsilon)\right|dx\leq\frac{C\varepsilon}\delta.$$
Thus on each connected component of $V$ the oscillation of $\Psi(w^\varepsilon)$ is controlled and the endpoints are controlled as well, by \eqref{F(w(x_i))<Ceps/delta}. Passing to the limit for $\varepsilon\rightarrow0$ in these estimates, we obtain \eqref{F(w(x)) tends 0}. Indeed, for example, let $\xi\in[a,x_1)$; since $\Psi'(s)=\sqrt{2F(s)}$ is positive except at $s=\pm1$, we have 
$$\frac{C\varepsilon}{\delta}\geq\left|\int_\xi^{x_1}\Psi'(w^\varepsilon)w^\varepsilon_xdx\right|=\left|\int_{w^\varepsilon(\xi)}^{w^\varepsilon(x_1)}\Psi'(s)ds\right|$$
and so
$$|\Psi(w^\varepsilon(x_1))-\Psi(w^\varepsilon(\xi))|\leq \frac{C\varepsilon}\delta,$$
for all $\xi\in[a,x_1)$. It follows that
\begin{equation}\label{|w(x_1)-w(xi)|}
\lim_{\varepsilon\rightarrow0}|w^\varepsilon(x_1)-w^\varepsilon(\xi)|=0,
\end{equation} 
for all $\xi\in[a,x_1)$. Combining  \eqref{F(w(x_i))<Ceps/delta} and \eqref{|w(x_1)-w(xi)|}, we obtain \eqref{F(w(x)) tends 0} for all $\xi\in[a,x_1)$.
\end{proof}
If $\delta\leq\bar\delta$ and hypotheses are satisfied, roughly speaking Lemma \ref{lemma-struttura-u^eps} asserts that $w^\varepsilon(x)\approx\pm1$ for small $\varepsilon$ if $|x-\gamma_i|\geq\delta$ for any $i=1,\ldots,N$. Indeed, \eqref{F(w(x)) tends 0} implies that $w^\varepsilon(x)\rightarrow\pm 1$ as $\varepsilon\rightarrow0$ for $x{\not\in\displaystyle\bigcup^N_{i=1}(x_i,y_i)}$. We use this Lemma to obtain results on the speed of the transition points. This is the subject of the following theorem.\par
Before stating the theorem let us recall some definitions. If $v$ is a step function with jumps at $\gamma_1,\gamma_2,\ldots,\gamma_N$, then its \emph{interface} $I[v]$ is defined by
$$I[v]:=\{\gamma_1,\gamma_2,\ldots,\gamma_N\}.$$
Now we fix some closed subset $K\subset\mathbb{R}\backslash\{\pm1\}$ and let $u:[a,b]\rightarrow\mathbb{R}$ be an arbitrary function. Then the \emph{interface} $I_K[u]$ is defined by
$$I_K[u]:=u^{-1}(K).$$
Finally, if $A,B\subset\mathbb{R}$, the \emph{Hausdorff distance} $d(A,B)$ between $A$ and $B$ is defined by $$d(A,B):=\max\left\{\sup_{\alpha\in A}d(\alpha,B),\,\sup_{\beta\in B}d(\beta,A)\right\}$$
where $d(\beta,A):=\inf\{|\beta-\alpha|: \alpha\in A\}$.
\begin{thm}\label{Teorema-distanza interfacce}
Assume that $F$ and $g$ satisfy \eqref{g strettamente positiva}-\eqref{hp-F-3} and that the initial data $u^\varepsilon_0$, $u^\varepsilon_1$ satisfy \eqref{u^eps tende a v} and \eqref{energia iniziale<} for some $k>0$. Let $u^\varepsilon$ solution of \eqref{allen-cahn-iper}-\eqref{Neumann-bordo}-\eqref{cond-iniz}. Given $\delta_1>0$ and a closed subset $K$ of $\mathbb{R}\backslash\{\pm1\}$, let
$$T_\varepsilon(\delta_1)=\inf\{t:\; d(I_K[u^\varepsilon(\cdot,t)],I_K[u^\varepsilon_0])>\delta_1\}.$$
If $\delta_1$ is sufficiently small, then for any $m>0$
\begin{equation}\label{T-distanza interfacce}
T_\varepsilon(\delta_1)> m\varepsilon^{-k}
\end{equation}
for $\varepsilon<\varepsilon_0(m,\delta_1)$.
\end{thm}
\begin{proof}
Let $\delta_1/2<\bar\delta$, where $\bar\delta$ is defined in \eqref{delta-bar}, so that we can apply Lemma \ref{lemma-struttura-u^eps}. Define 
$$L_i:=\left(\gamma_i-\frac{\delta_1}2,\gamma_i+\frac{\delta_1}2\right)\qquad i=1,\ldots,N \qquad \mbox{ and }\quad L:=L_1\cup\cdots\cup L_N.$$ 
Furthermore, choose $\varrho>0$ such that $(-1-\varrho,-1+\varrho)$ and $(1-\varrho,1+\varrho)$ are contained in $\mathbb{R}\backslash K$.
From hypotheses \eqref{u^eps tende a v}, \eqref{energia iniziale<} and Lemma \ref{lemma-struttura-u^eps}, it follows that for sufficiently small  $\varepsilon>0$ we have 
$$d\left(\{u^\varepsilon_0(x)\; : \; x\in[a,b]\backslash L\} ,\{\pm1\}\right)<\varrho.$$ 
This implies the inclusion $I_K[u^\varepsilon_0]\subset L$ and therefore $$d(I_K[u^\varepsilon_0],I[v])\leq\delta_1/2.$$
By Theorem \ref{metastability-thm} and the fact that $E_\varepsilon[u^\varepsilon,u^\varepsilon_t](t)$ is decreasing in $t$, we can apply Lemma \ref{lemma-struttura-u^eps} to $w=u^\varepsilon(\cdot,T)$ for all $T\leq m\varepsilon^{-k}$ if $\varepsilon$ is sufficiently small. Therefore we obtain $d(I_K[u^\varepsilon(\cdot,T)],I[v])\leq\delta_1/2$. Using triangle inequality we infer for sufficiently small $\varepsilon>0$
$$d(I_K[u^\varepsilon_0],I_K[u^\varepsilon(\cdot,T])\leq\delta_1$$
for any $T\leq m\varepsilon^{-k}$.
\end{proof}

\section{Numerical explorations}\label{numerical example}
In this Section we present some numerical solutions of \eqref{Cauchy problem} showing dynamical metastability. We consider the case $g(u)=1-\tau f'(u)$.\par
As previously mentioned, equation \eqref{allen-cahn-iper-g=1-tau f'} has a probabilistic interpretation and appears in the description of a \emph{correlated random walk}. Following Taylor \cite{Tay}, Goldstein \cite{Gol}, Kac \cite{Kac}, Zauderer \cite{Zau} and Holmes \cite{Holmes}, we consider particles moving along a line. We assume that the particles make steps of length $dx$ and duration $dt$. At all times particles move with velocity $\gamma=\frac{dx}{dt}$ and at any step a particle continues in its previous direction with probability $p$ and reverses direction with probability $q$. For small $dt$, $p=1-\lambda dt$ and $q=\lambda dt$, where $\lambda$ is the rate of reversal. The reversal process can be thought as a Poisson process with intensity $\lambda$. We split the particle density $u(x,t)=\alpha(x,t)+\beta(x,t)$, where $\alpha(x,t)$ and $\beta(x,t)$ are the densities, at coordinate $x$ at time $t$, of particles that arrived from the left and right, respectively. Finally, we assume that in the time interval $[t,t+dt]$, $dt f(u(x,t))$ particles are produced in $x$ and the new particles have equal probability of going left or right. It follows that
\begin{align}
\alpha(x,t+dt)=p\alpha(x-dx,t)+q\beta(x-dx,t)+\frac12dtf(u(x-dx,t)), \label{diff-finite alfa}\\
\beta(x,t+dt)=p\beta(x+dx,t)+q\alpha(x+dx,t)+\frac12dtf(u(x+dx,t)). \label{diff-finite beta}
\end{align}
We substitute $p=1-\lambda dt$, $q=\lambda dt$ and $dx=\gamma dt$, use the Taylor series to expand $\alpha,\beta$ and take the limit as $dt$ goes to zero to obtain the \emph{nonlinear Goldstein-Kac system}
\begin{align}
\alpha_t+\gamma\alpha_x=\lambda(\beta-\alpha)+\frac12f(\alpha+\beta) \label{Goldstein-Kac1},\\
\beta_t-\gamma\beta_x=\lambda(\alpha-\beta)+\frac12f(\alpha+\beta)\label{Goldstein-Kac2}.
\end{align}
Written in terms of the particle density $u$ and the particle flow $v:=\alpha-\beta$ this systems reads
\begin{align}
u_t+\gamma v_x &= f(u), \label{cat1}\\
v_t+\gamma u_x &=-2\lambda v. \label{cat2}
\end{align}
From this system, it follows that $u$ is a solution of the equation
$$u_{tt}+(2\lambda-f'(u))u_t=\gamma^2u_{xx}+2\lambda f(u),$$
that is \eqref{allen-cahn-iper-g=1-tau f'} with $\tau=\frac1{2\lambda}$ and $\varepsilon=\frac\gamma{\sqrt{2\lambda}}$.\par
Let us use solutions of the nonlinear Goldstein-Kac system \eqref{Goldstein-Kac1}-\eqref{Goldstein-Kac2} to construct solutions of \eqref{Cauchy problem}. We define the boundary conditions for system \eqref{Goldstein-Kac1}-\eqref{Goldstein-Kac2}. Assume that no particle can leave the interval $[a,b]$. Hence particles are reflected and the boundary conditions for the system \eqref{Goldstein-Kac1}-\eqref{Goldstein-Kac2} reads
\begin{equation}\label{GK-bordo}
\alpha(a,t)=\beta(a,t),\qquad\; \beta(b,t)=\alpha(b,t) \qquad \;\forall t\geq0.
\end{equation}
The corresponding boundary conditions for the system \eqref{cat1}-\eqref{cat2} are 
$$v(a,t)=v(b,t)=0 \qquad \quad \forall\,t\geq0.$$ 
Therefore, from \eqref{cat2} it follows that $u$ satisfies homogeneous Neumann boundary conditions \eqref{Neumann-bordo}. Moreover, calculating \eqref{cat1} in $t=0$, we obtain
\begin{equation}\label{initial data GK-allen cahn iper}
\begin{cases}
\displaystyle\alpha_0(x)+\beta_0(x) =  u_0(x) \\
\displaystyle\alpha_0(x)-\beta_0(x) = \dfrac1\gamma\int_a^x [f(u_0(s))-u_1(s)]ds
\end{cases},
\end{equation}
where $\alpha_0(x)=\alpha(x,0)$ and $\beta_0(x)=\beta(x,0)$. In conclusion, we can say that if $(\alpha,\beta)$ is a solution sufficiently regular of \eqref{Goldstein-Kac1}-\eqref{Goldstein-Kac2} with 
$$\lambda=\frac1{2\tau}\, \qquad \mbox{ and } \qquad \gamma=\frac{\varepsilon}{\sqrt{\tau}},$$
satisfying boundary conditions \eqref{GK-bordo} and initial data satisfying \eqref{initial data GK-allen cahn iper}, then $u=\alpha+\beta$ is solution of \eqref{Cauchy problem} with $g(u)=1-\tau f'(u)$.\par
From \eqref{GK-bordo} and \eqref{initial data GK-allen cahn iper}, it follows the compatibility condition for the initial data
\begin{equation}\label{condition initial data}
\int_a^b[f(u_0(s))-u_1(s)]ds=0.
\end{equation}
Given $(u_0,u_1)\in H^2([a,b])\times H^1([a,b])$ satisfying \eqref{condition initial data}, the corresponding initial data for system \eqref{Goldstein-Kac1}-\eqref{Goldstein-Kac2} are  
\begin{equation}\label{alfa_0,beta_0}
\begin{cases}
\displaystyle\alpha_0(x) =\dfrac12\left(u_0(x)+\dfrac1\gamma\int_a^x [f(u_0(s))-u_1(s)]ds\right) \\
\displaystyle\beta_0(x) = \dfrac12\left(u_0(x)-\dfrac1\gamma\int_a^x [f(u_0(s))-u_1(s)]ds\right)
\end{cases}
\end{equation}
In this Section we use a finite difference method based on system \eqref{diff-finite alfa}-\eqref{diff-finite beta} with $p=1-\lambda dt$, $q=\lambda dt$ and $dx=\gamma dt$ to calculate numerical solutions of \eqref{Goldstein-Kac1}-\eqref{Goldstein-Kac2}, with boundary conditions \eqref{GK-bordo} and initial data \eqref{alfa_0,beta_0}. The sum $u=\alpha+\beta$ is the solution of \eqref{Cauchy problem}.\par
Let us consider $f(u)=u-u^3$ and so $g(u)=1-\tau(1-3u^2)$. Let $[a,b]=[-4,4]$. \\\par

\textbf{Example 1}. In the first example, we choose $u_0(x)=\cos(\frac\pi2 x)/10$ and $u_1=0$. The condition \eqref{condition initial data} is satisfied. The numerical solutions for different times $t$ are shown in Figure \ref{u0=cos eps=0.01}. The initial profile $u_0$ has $4$ zeros and takes values in $[-0.1,0.1]$. In a short time a metastable state is formed: in the intervals where $u_0>0$ ($u_0<0$), the solution $u$ reaches the value $1$ ($-1$) in a short time $t$ and so we have $4$ transitions between $1$ and $-1$. In a longer time scale the solution evolves very slowly and appears to be stable. In this example, $u_1=0$ and the qualitative behavior of the solution is the same as the parabolic case \eqref{Allen-Cahn}.

\begin{figure}[htbp]
\centering
\includegraphics[scale=.4]{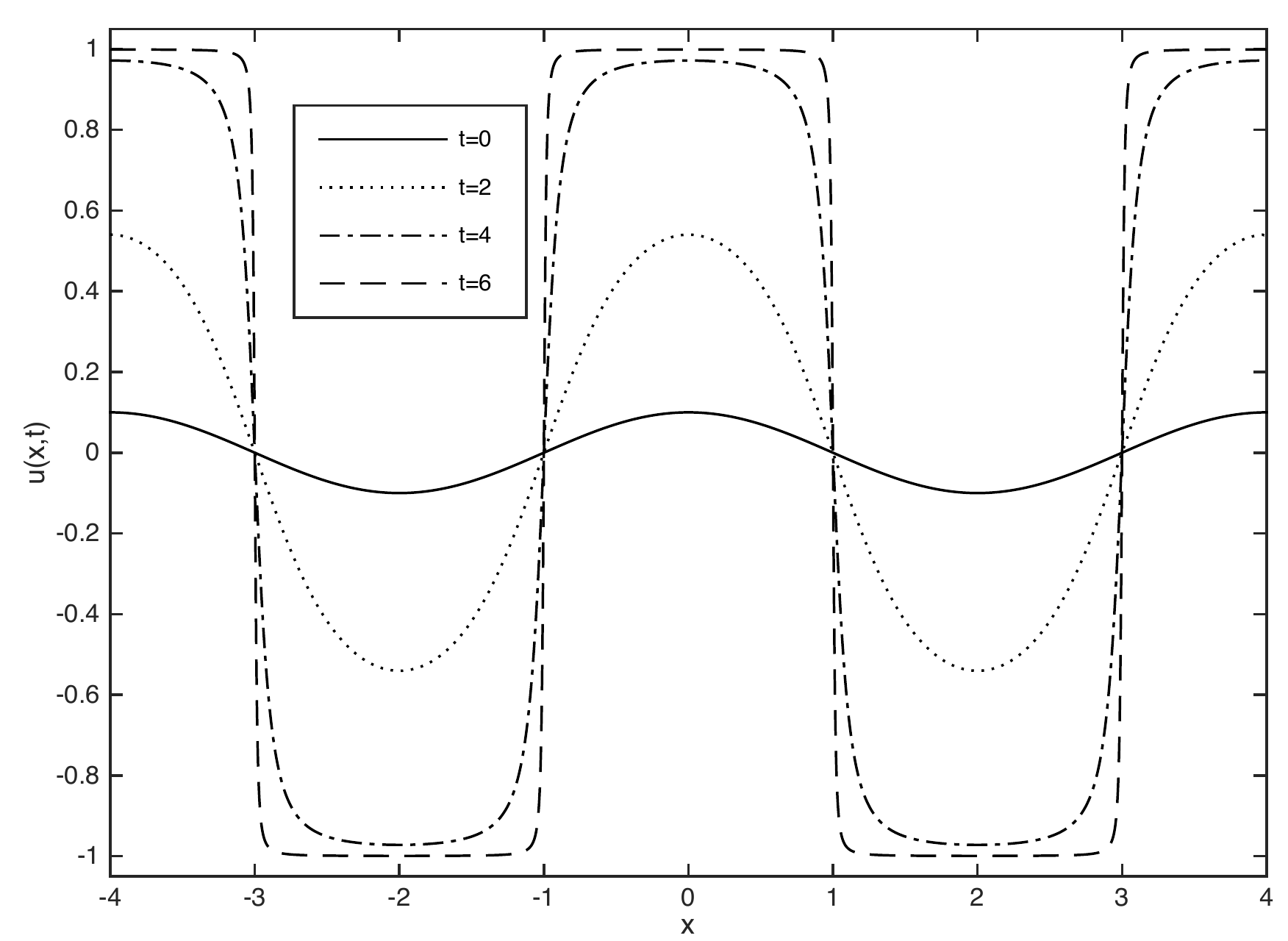}%
\includegraphics[scale=.4]{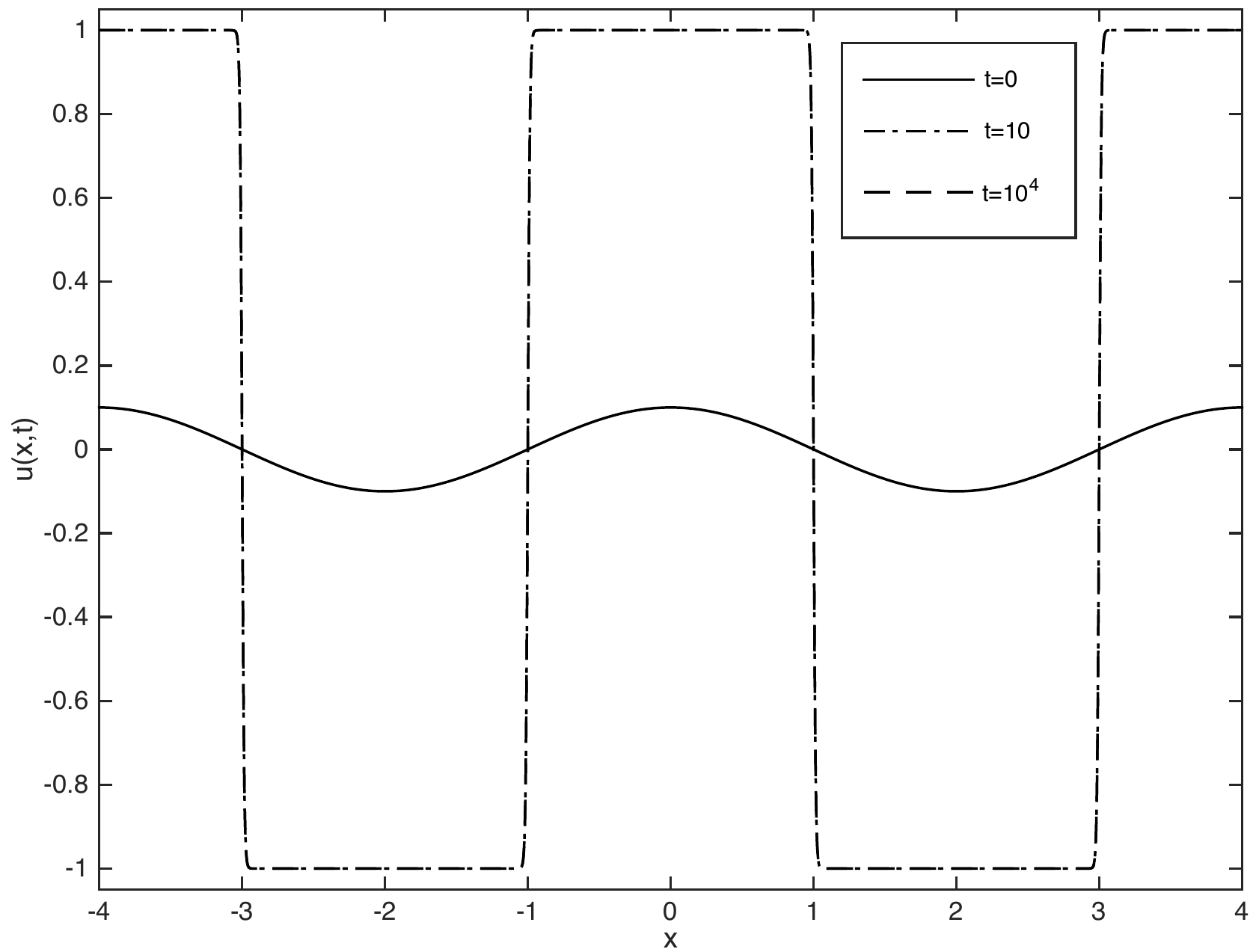}
\caption{Initial data: $u_0(x)=\cos(\frac\pi2 x)/10$, $u_1(x)=0$. The values of constants are: $\varepsilon=0.01, \tau=0.8$.}
\label{u0=cos eps=0.01}
\end{figure}

\textbf{Example 2}. Let $u_0(x)=0$ for all $x\in[-4,4]$ and $u_1(x)=\cos(\frac\pi2 x)$.  The condition \eqref{condition initial data} is satisfied; the numerical solutions are shown in Figure \ref{u0=0 eps=0.1}. Even if the initial profile $u_0$ is identically zero, a metastable state is created. The number of transitions between $1$ and $-1$ is equal to the number of sign changes of $u_1$. This is a simple example where $u_0$ has not transitions, but the initial velocity $u_1$ \emph{creates} a metastable state.

\begin{figure}[htbp]
\centering
\includegraphics[scale=.4]{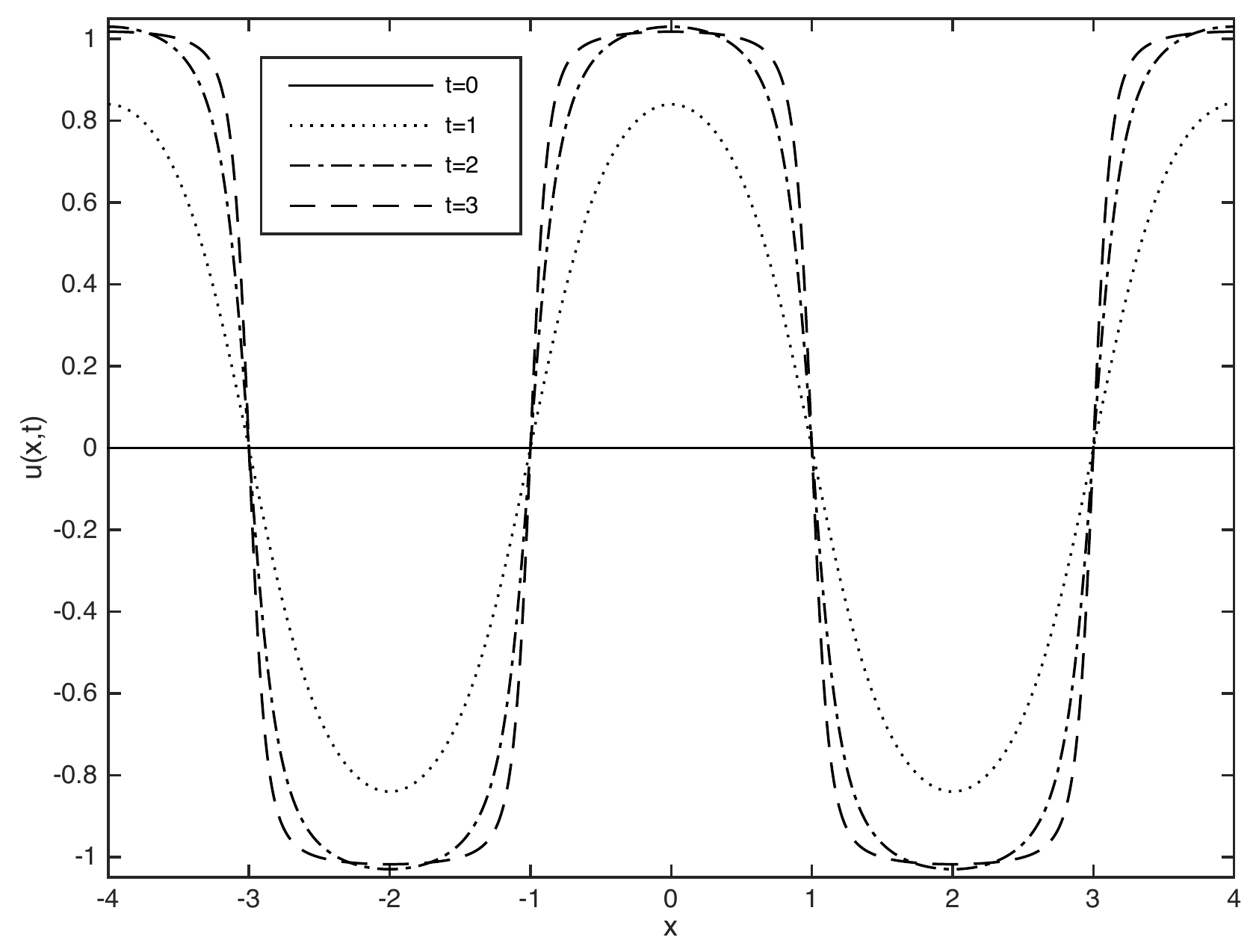}%
\includegraphics[scale=.4]{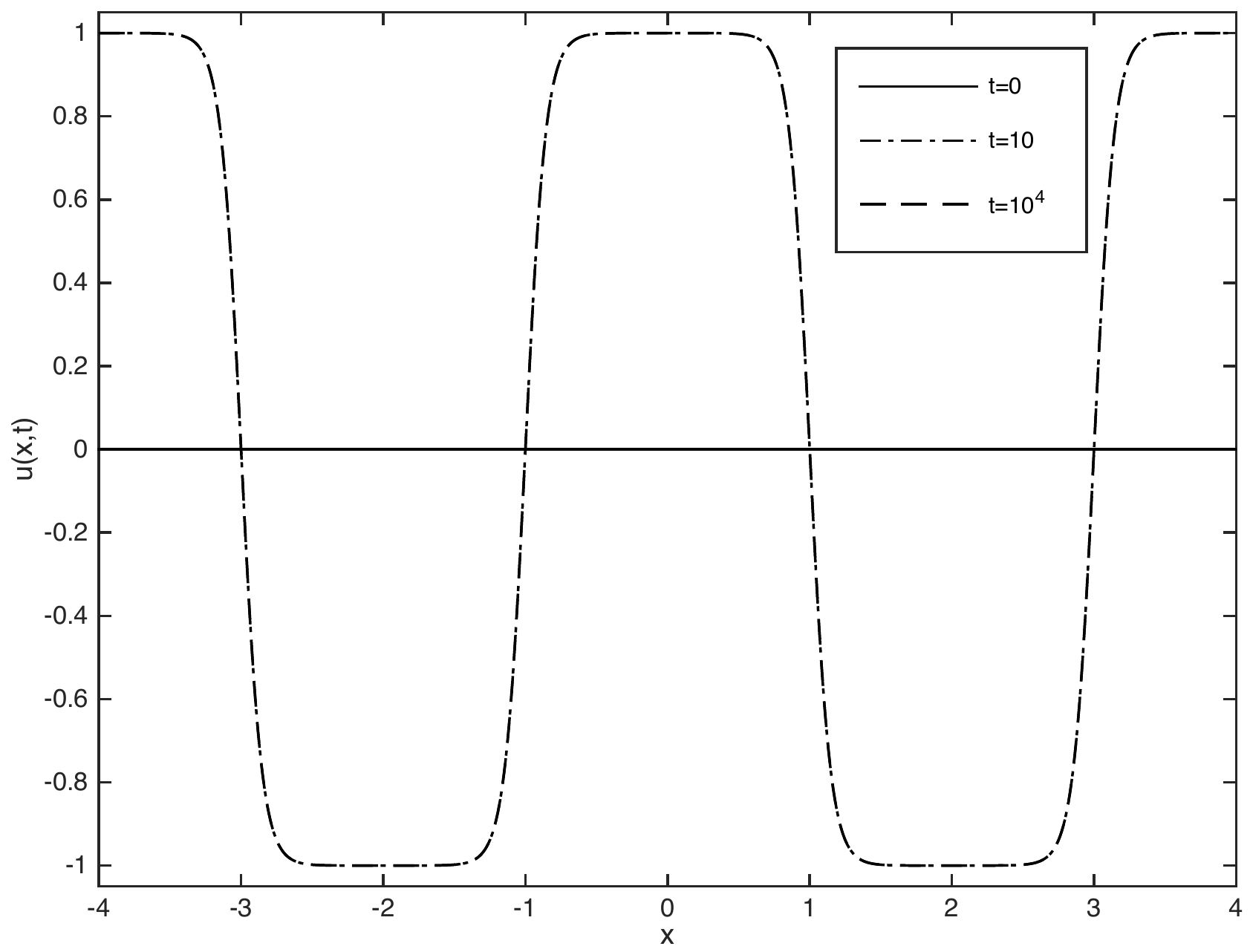}
\caption{Initial data: $u_0(x)=0$, $u_1(x)=\cos(\frac\pi2 x)$. The values of constants are: $\varepsilon=0.1, \tau=0.8$.}
\label{u0=0 eps=0.1}
\end{figure}

In the first two examples $u_0$ does not verify the hypotheses of Theorem \ref{metastability-thm}. In the following examples we will consider initial profiles $u_0$ verifying hypotheses of Theorem \ref{metastability-thm}.\par
\textbf{Example 3}. We consider an initial profile $u_0$ that satisfies the hypotheses of Theorem \ref{metastability-thm}. To do this, we use a travelling wave solution of \eqref{allen-cahn-iper-g=1-tau f'} connecting $-1$ and $1$, i.e. a solution of the problem
\begin{equation}\label{sol-staz-eteroclina}
\begin{cases}
\varepsilon^2u_{xx}+u-u^3=0 \qquad \qquad x\in\mathbb{R}\\
\displaystyle\lim_{x\rightarrow-\infty}u(x)=-1 \qquad  \lim_{x\rightarrow+\infty}u(x)=1.
\end{cases}
\end{equation}
A solution of \eqref{sol-staz-eteroclina} is 
\begin{equation}\label{tanh}
\phi^\varepsilon(x)=\tanh\left(\frac x{\sqrt2\varepsilon}\right).
\end{equation}
We have
$$\lim_{\varepsilon\rightarrow0}\phi^\varepsilon(x)=\begin{cases}
-1 \qquad \,\mbox{ if }\; x<0\\
1 \qquad\quad \mbox{ if } \; x>0
\end{cases} \qquad \quad \mbox{ in } \, L^1(a,b)$$
and
$$\lim_{\varepsilon\rightarrow0}\int_a^b\left[\frac\varepsilon2 \left(\phi^\varepsilon_x\right)^2+\frac1{4\varepsilon}\left((\phi^\varepsilon)^2-1\right)^2\right]dx=\frac{2\sqrt2}3=:c_0,$$
for all $a<0<b$. Let $[a,b]=[-4,4]$, $u_0(x)=\displaystyle\tanh\left(\frac x{\sqrt2 \varepsilon}\right)$ and $u_1(x)=-x$. The initial data $u_0$ and $u_1$ are odd functions and so the condition \eqref{condition initial data} is satisfied. The numerical solutions are shown in Figure \ref{u0=tanh eps=0.2}. The initial profile $u_0$ has a transition layer structure, with a transition at $x=0$, but the initial velocity $u_1$ in a short time \emph{creates} a metastable state with $3$ transitions.\par
\begin{figure}[htbp]
\centering
\includegraphics[scale=.4]{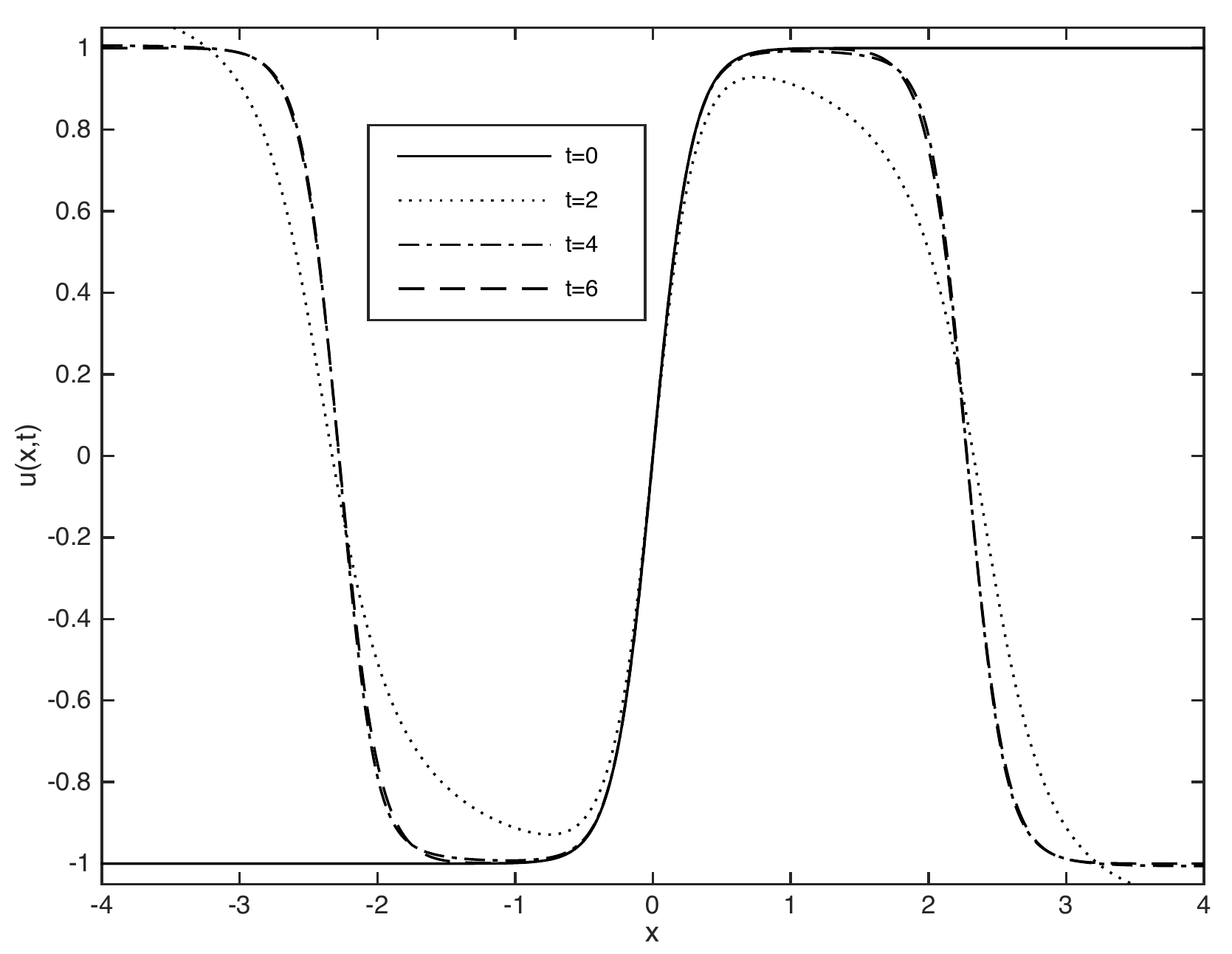}%
\includegraphics[scale=.4]{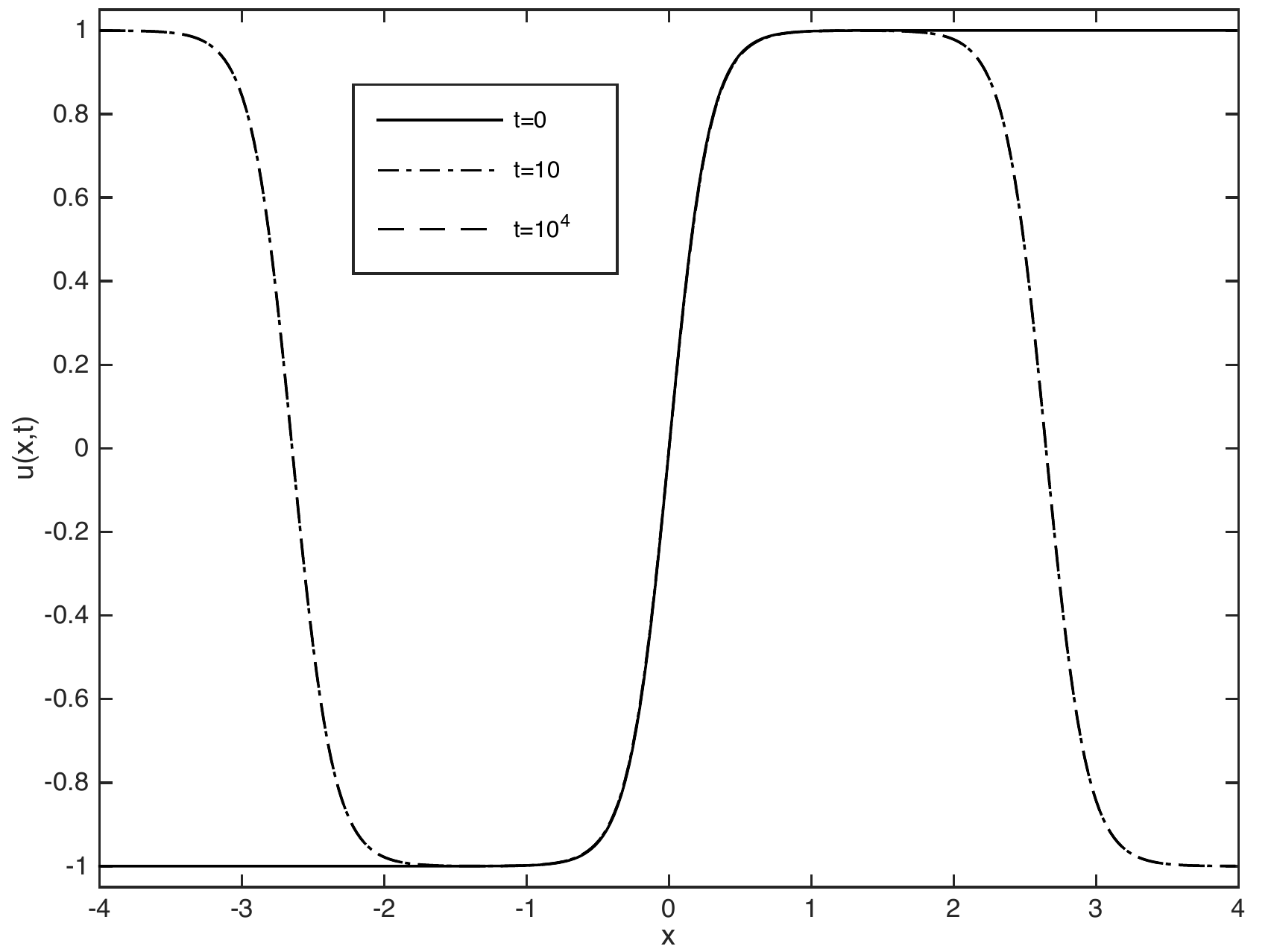}
\caption{Initial data: $u_0(x)=\tanh\left(\frac{x}{\sqrt2\varepsilon}\right)$, $u_1(x)=-x$. The values of constants are: $\varepsilon=0.2, \tau=0.6$.}
\label{u0=tanh eps=0.2}
\end{figure}

\textbf{Example 4}. In conclusion, we show an example where there is a loss of transition. Using the function \eqref{tanh}, we construct an initial profile with $2$ transitions between $1$ and $-1$. Let us consider
$$u_0(x)=\begin{cases}\tanh\left(\displaystyle\frac{x+2}{\sqrt2\varepsilon}\right) \qquad -4\leq x\leq0\\
-\tanh\left(\displaystyle\frac{x-2}{\sqrt2\varepsilon}\right) \qquad 0\leq x\leq4.
\end{cases}$$
The numerical solutions are shown in Figure \ref{loss transition eps=0.01}. Even if the initial profile $u_0$ has $2$ transitions and $\varepsilon$ is small, thanks to the initial velocity $u_1(x)=-x$, a metastable state with one transition is formed.

\begin{figure}[htbp]
\centering
\includegraphics[scale=.4]{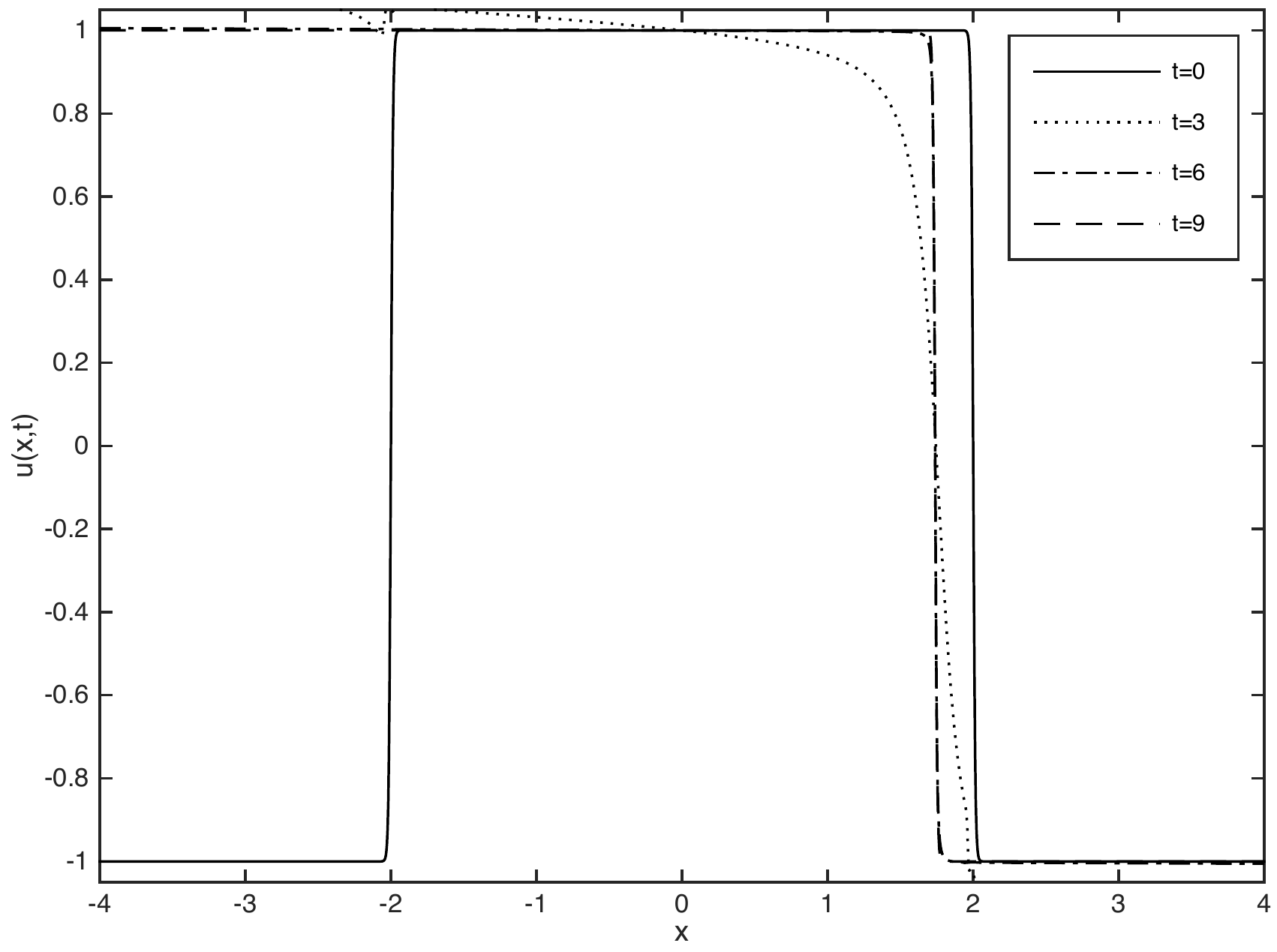}%
\includegraphics[scale=.4]{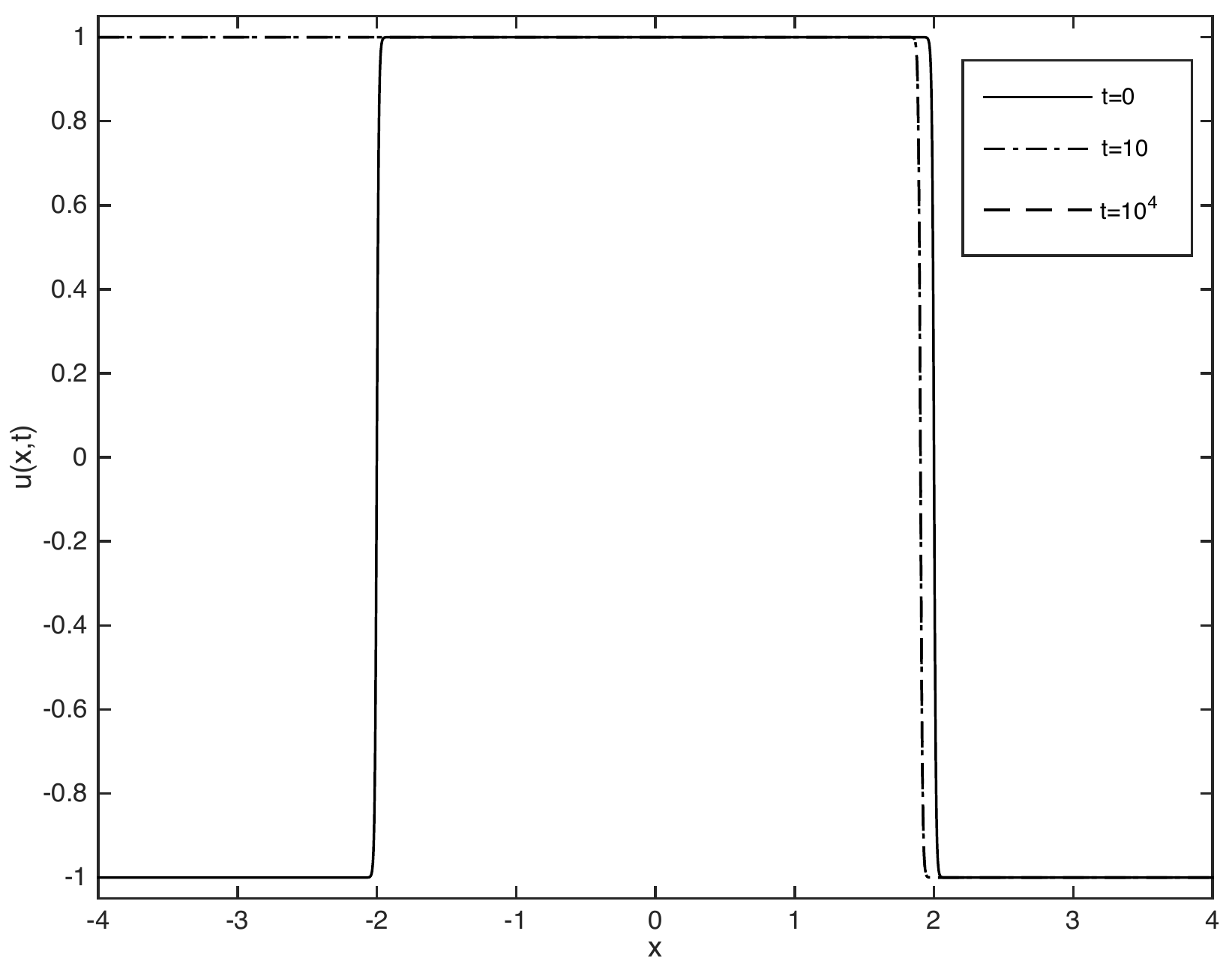}
\caption{Initial data: $u_0(x)$ with two transitions and $u_1(x)=-x$. The values of constants are: $\varepsilon=0.01, \tau=0.9$.}
\label{loss transition eps=0.01}
\end{figure}

\begin{appendices}
\section{Existence and uniqueness}\label{exist-uniq}
In this Appendix we study the problem of existence and uniqueness of solutions of the equation
\begin{equation}\label{allen-cahn-iper-app}
\tau u_{tt}+g(u)u_t=\varepsilon^2u_{xx}+f(u), \qquad \quad (x,t)\in [a,b]\times(0,T),
\end{equation}
with homogeneous Neumann boundary conditions
\begin{equation}\label{Neumann-bordo-app}
u_x(a,t)=u_x(b,t)=0 \qquad\quad \forall t>0
\end{equation}
and initial data 
\begin{equation}\label{cond-iniz-app}
u(x,0)=u_0(x), \qquad u_t(x,0)=u_1(x), \qquad \qquad x\in[a,b].
\end{equation} 
We use the semigroup theory for solutions of differential equations on Hilbert spaces. Specifically, we recall (see Pazy \cite{Pazy}) that, given a Hilbert space $X$, a linear operator $A: D(A)\subset X\rightarrow X$ is \emph{m}-dissipative if 
\begin{description}
\item[i)] A is dissipative, i.e. $\langle Au,u\rangle_X\leq0$, for all $u\in D(A)$;
\item[ii)] for all $f\in X$, there exists $u\in D(A)$ such that $u-Au=f$.
\end{description}
For the sake of completeness, we recall two results on the \emph{m}-dissipative operators. 
\begin{thm}[Lumer-Phillips Theorem]
A linear operator $A$ is the generator of a contraction semigroup $(S(t))_{t\geq0}$ in $X$ if and only if $A$ is m-dissipative with dense domain.
\end{thm}
\begin{lem}\label{dense-dom}
Let $X$ be a Hilbert space. If $A: D(A)\subset X\rightarrow X$ is m-dissipative, then $D(A)$ is dense in $X$.
\end{lem}
\begin{proof}
Let $x\in X$ such that $\langle x,z\rangle_X=0$ for all $z\in D(A)$ and let $u\in D(A)$ such that $u-Au=x$. Then 
$$0=\langle x,u\rangle_X=\langle u-Au,u\rangle_X.$$
Hence, $\|u\|^2_X=\langle Au,u\rangle_X\leq0$. It follows that $u=x=0$ and so $D(A)$ is dense in $X$.
\end{proof}
Setting $\y=(u,v)=(u,\partial_t u)$, we rewrite \eqref{allen-cahn-iper-app} as a first order evolution equation 
\begin{equation}\label{eq-primo-ordine}
\y_t=A\y+\Phi(\y),
\end{equation}
where 
\begin{equation}\label{A-Fi}
A\y:=\left(\begin{array}{cc} 0 & 1 \\ \varepsilon^2\tau^{-1}\partial^2_x & 0 \end{array} \right)\y-\y\qquad \mbox{and}\qquad \Phi(\y):=\y+\dfrac1\tau\left(\begin{array}{c} 0 \\ f(u)-g(u)v \end{array}\right).
\end{equation}
The unknown $\y$ is considered as a function of a real (positive) variable $t$ with values on the function space $X=H^1([a,b])\times L^2(a,b)$ with scalar product
$$\langle(u,v),(w,z)\rangle_X:= \int_a^b (\varepsilon^2u_xw_x+\tau uw+\tau vz)dx,$$ 
that is equivalent to the usual scalar product in $H^1([a,b])\times L^2(a,b)$. \par
We use the notation $H^k:=H^k([a,b])$ and $L^k:=L^k(a,b)$, for $k=1,2$.
\begin{prop}\label{A-m-dissipativo}
The linear operator $A:D(A)\subset X\rightarrow X$ defined by \eqref{A-Fi} with 
\begin{equation}\label{dominio-A}
D(A)=\left\{(u,v)\in H^2\times H^1 : u_x(a)=u_x(b)=0\right\},
\end{equation}
is m-dissipative with dense domain.
\end{prop}
\begin{proof}
For all $\y=(u,v)\in D(A)$, using integration by parts, we obtain
$$\langle A\y,\y\rangle_X = \int_a^b\left(-\varepsilon^2u_x^2-\tau \left(u^2-uv+v^2\right)\right)dx\leq0.$$
Hence $A$ is dissipative. Now, let us show that for all $\x=(h,k)\in X$ there exists $\y=(u,v)\in D(A)$ such that $\y-A\y=\x$. This equation is equivalent to the system
\begin{equation}\label{sistema-u-v}
\begin{cases}
2u-v=h \\
2v-\frac{\varepsilon^2}\tau u_{xx}=k 
\end{cases},
\end{equation}
with $h\in H^1$ and $k\in L^2$. By \eqref{sistema-u-v} it follows that $u$ satisfies 
\begin{equation}\label{u-ellittico}
4u-\frac{\varepsilon^2}\tau u_{xx}=2h+k.
\end{equation}
Hence we must show that there exists $u\in H^2$ with $u_x(a)=u_x(b)=0$ satisfying \eqref{u-ellittico}. The weak formulation of \eqref{u-ellittico} is given by
\begin{equation}\label{u-ellittico-debole}
\alpha(u,w):=4\tau\int_a^buw\, dx+\varepsilon^2\int_a^bu_x w_x\,dx=\tau\int_a^b(2h+k)w\,dx=:\varphi(w), 
\end{equation}
for all $w\in H^1$. Applying Lax-Milgram Theorem in $H^1$ to the bilinear form $\alpha$ and the linear form $\varphi$, we obtain that there exists a solution $u\in H^1$ of \eqref{u-ellittico-debole}. \par
Identity \eqref{u-ellittico-debole} gives also
$$\int_a^bu_xw_x\,dx=\frac{\tau}{\varepsilon^2}\int_a^b(2h+k-4u)w\,dx \qquad \quad \forall\, w\in C^1_c([a,b])$$
and so, since $2h+k-4u\in L^2$, we have $u\in H^2$.\par
Finally, we have to show that the boundary conditions are satisfied. To this aim, we observe that
\begin{equation}\label{u-debole-secondo}
\int_a^b(4\tau u-\varepsilon^2u_{xx}-\tau(2h+k))w\,dx+\varepsilon^2(u_x(b)w(b)-u_x(a)w(a))=0 
\end{equation}
for all $w\in H^1$. By choosing $w\in H^1_0([a,b])$ in \eqref{u-debole-secondo}, we obtain that $u$ satisfies \eqref{u-ellittico} almost everywhere and therefore $u_x(b)w(b)-u_x(a)w(a)=0$ for all $w\in H^1$. It follows that $u_x(a)=u_x(b)=0$.\par
Next, solving the first equation of the system \eqref{sistema-u-v}, we obtain $v\in H^1$. Therefore, $A$ is \emph{m}-dissipative. The domain $D(A)$ is dense for Lemma \ref{dense-dom}.
\end{proof}
From Proposition \ref{A-m-dissipativo} and Lumer-Phillips Theorem it follows that the operator $A:D(A)\subset X\rightarrow X$ defined by \eqref{A-Fi}-\eqref{dominio-A} is the generator of a contraction semigroup $(S(t))_{t\geq0}$ in $X$.\par
Now, we study the well-posedness of the Cauchy problem for the semilinear equation \eqref{eq-primo-ordine}. To do this, we use some results from Cazenave and Haraux \cite[Chapter 4]{Cazenave}.\par
In the following, we suppose that $\Phi:X\rightarrow X$ is a Lipschitz continuous function on bounded subsets of $X$. We denote by $L(M)$ the Lipschitz constant of $\Phi$ in $B_M$ for $M>0$, where $B_M$ is the ball of center $0$ and of radius $M$.\par
Given $\x\in X$, we look for $T>0$ and a classical solution 
$$\y\in C([0,T],D(A))\cap C^1([0,T],X)$$ 
of the problem:
\begin{equation}\label{prob-Cauchy-classical}
\begin{cases}
\y_t(t)=A\y(t)+\Phi(\y(t)),\qquad \forall\,t\in[0,T];\\
\y(0)=\x.
\end{cases}
\end{equation}
It can be shown that any classical solution $\y$ of \eqref{prob-Cauchy-classical} is also a \emph{mild solution} on $[0,T]$, that is a function $\y\in C([0,T],X)$ solving the problem 
\begin{equation}\label{weak-form}
\y(t)=S(t)\x+\int_0^t S(t-s)\Phi(\y(s))ds, \qquad \forall\,t\in[0,T].
\end{equation}
The following result states that such a solution exists and it is unique for any $\x\in X$. 
\begin{thm}[see Cazenave-Haraux \cite{Cazenave}, Theorem $4.3.4$]\label{esist-unic-altern}
There exists a function $T:X\rightarrow(0,\infty]$ with the following properties: for all $\x\in X$, there exists ${\y\in C([0,T(\x)),X)}$ such that for all $0<T<T(\x)$, $\y$ is the unique solution of \eqref{weak-form} in $C([0,T],X)$. In addition, 
$$2L(\|\Phi(0)\|_X+2\|\y(t)\|_X)\geq \frac{1}{T(\x)-t}-2,$$
for all $t\in[0,T(\x))$. In particular, we have the following alternatives:
\begin{equation}\label{alternatives-T(x)}
(a) \quad T(\x)=\infty \qquad \mbox{or} \qquad (b)\quad T(\x)<\infty \quad \mbox{and}\quad\lim_{t\uparrow T(\x)}\|\y(t)\|_X=\infty.
\end{equation}
\end{thm}
From Theorem \ref{esist-unic-altern} it follows that for all $\x\in X$ the problem \eqref{prob-Cauchy-classical} has a unique mild solution $\y\in C([0,T(\x)),X)$. Regarding the regularity of the solutions, we have that if $\x\in D(A)$, then $\y$ is a classical solution (see Cazenave and Haraux \cite[Proposition 4.3.9]{Cazenave}). Finally, the solution $\y(t)$ depends continuously on the initial data $\x\in X$, uniformly for all $t\in[0,T]$:
\begin{prop}[\cite{Cazenave}, Proposition 4.3.7]\label{dip-con-dat-iniz}
Following the notation of Theorem \ref{esist-unic-altern}, we have the following properties:
\begin{description}
\item[i)] $T: X\rightarrow (0,\infty]$ is lower semicontinuous;
\item[ii)] if $\x_n\rightarrow\x$ in $X$ and if $T<T(\x)$, then $\y_n\rightarrow \y$ in $C([0,T],X)$, where $\y_n$ and $\y$ are the solutions of \eqref{weak-form} corresponding to the initial data $\x_n$ and $\x$.
\end{description}
\end{prop}
In order to apply Theorem \ref{esist-unic-altern} and Proposition \ref{dip-con-dat-iniz}, the function $\Phi$ defined by \eqref{A-Fi} must be a Lipschitz continuous function on bounded subsets of $X$. This is guaranteed if $f,g$ are locally Lipschitz continuous on $\mathbb{R}$.\par
Indeed, for all $\y_1=(u_1,v_1)$, $\y_2=(u_2,v_2)\in X$ we have
\begin{align*}
\|\Phi(\y_1)-\Phi(\y_2)\|_X \leq &\; \|\y_1-\y_2\|_X \\
& +\tau^{-1/2}\left(\|f(u_1)-f(u_2)\|_{L^2}+\|g(u_1)v_1-g(u_2)v_2\|_{L^2}\right).
\end{align*}
Let $M:=\max\{\|\y_1\|_X,\|\y_2\|_X\}$. We have that
\begin{align*}
\|g(u_1)v_1-g(u_2)v_2\|_{L^2}\leq\,& \|g(u_1)(v_1-v_2)\|_{L^2}+\|v_2(g(u_1)-g(u_2))\|_{L^2}\\
\leq\,& \|g(u_1)\|_{L^\infty}\|v_1-v_2\|_{L^2}+C_2\|u_1-u_2\|_{H^1},
\end{align*}
where the last inequality holds because $g$ is locally Lipschitz continuous and $H^1([a,b])\subset L^\infty([a,b])$ with continuous inclusion. From this inequality and $\|f(u_1)-f(u_2)\|_{L^2}\leq L\|u_1-u_2\|_{L^2}$, it follows that there exists a constant $C(M)$ (depending on $M$) such that
$$\|\Phi(\y_1)-\Phi(\y_2)\|_X\leq C(M)\|\y_1-\y_2\|_X.$$
Therefore, we can say that for all $\x\in X$ the problem \eqref{prob-Cauchy-classical}, with $A$ and $\Phi$ defined by \eqref{A-Fi}-\eqref{dominio-A} and $f,g$ locally Lipschitz continuous, has a unique mild solution on $[0,T(\x))$. In particular, if $\x\in D(A)$, then $\y$ is a classical solution.\par 
Now, let us seek assumptions on $f$ and $g$ such that any solution is global, i.e. $T(\x)=\infty$ for all $\x\in X$. The alternatives $(a)-(b)$ of \eqref{alternatives-T(x)} mean that the global existence of the solution $\y=(u,u_t)$ is equivalent to the existence of an \emph{a priori} estimate of $\|(u,u_t)\|_X$ on $[0,T(\x))$. We show that, under appropriate assumptions on $f$ and $g$, $(b)$ cannot occur by using energy estimates. \par
We define the energy
\begin{equation}\label{energy}
E[u,u_t](t):=\int_a^b\left[\frac\tau2 u^2_t(x,t)+\frac{\varepsilon^2}2 u^2_x(x,t)+F(u(x,t))\right]dx,
\end{equation}
where $\displaystyle F(u)=-\int_0^u f(s)ds$. Observe that the energy \eqref{energy} is well-defined for mild solutions $(u,u_t)\in C([0,T],X)$.
\begin{prop}\label{prop-var-ener}
We consider the problem \eqref{prob-Cauchy-classical} with $A$ and $\Phi$ defined by \eqref{A-Fi}-\eqref{dominio-A} and $f,g$ locally Lipschitz continuous. If $\y=(u,u_t)\in C([0,T],X)$ is a mild solution, then 
\begin{equation}\label{variazione-energia}
\int_0^T\int_a^b g(u)u^2_t dxdt=E[u,u_t](0)-E[u,u_t](T).
\end{equation}
\end{prop}
\begin{proof}
Let $T>0$ and $\y=(u,u_t)$ a classical solution of \eqref{prob-Cauchy-classical}. Then $u(x,t)$ is a classical solution of \eqref{allen-cahn-iper-app} with boundary conditions \eqref{Neumann-bordo-app}; we multiply \eqref{allen-cahn-iper-app} by $u_t$ and integrate on $[a,b]\times[0,T]$:
$$\int_0^T\int_a^b\left(\tau u_tu_{tt}+g(u)u^2_t\right)dxdt=\int_0^T\int_a^b\left(\varepsilon^2 u_tu_{xx}+f(u)u_t\right)dxdt.$$
Using integration by parts and the boundary conditions \eqref{Neumann-bordo-app} we obtain
\begin{align*}
\int_0^T\int_a^b g(u)u^2_t dxdt &=\int_a^b\left[\frac\tau2 u^2_t(x,0)-\frac\tau2u^2_t(x,T)\right]dx\\
&+\int_a^b\left[\frac{\varepsilon^2}2 u^2_x(x,0)-\frac{\varepsilon^2}2 u^2_x(x,T)\right]dx\\
&+\int_a^b\left[F(u(x,0))-F(u(x,T))\right]dx. 
\end{align*}
Using the definition of energy \eqref{energy} we have \eqref{variazione-energia} for classical solution.\par
If $\x\in D(A)$ the solution is classical and \eqref{variazione-energia} holds.   If $\x\in X\backslash D(A)$, we consider $\x_n\in D(A)$ such that $\x_n\rightarrow \x$ in $X$. For the corresponding solution $\y_n=(u_n,(u_n)_t)$, \eqref{variazione-energia} is satisfied; for Proposition \ref{dip-con-dat-iniz}, by passing to the limit we obtain \eqref{variazione-energia} for $\y=(u,u_t)$.
\end{proof}
In the following, we consider mild solutions. If we assume that $g(u)\geq0$ for all $u\in\mathbb{R}$, then the energy defined by \eqref{energy} is a nonincreasing function of $t$ along the solutions of \eqref{allen-cahn-iper-app} with boundary conditions \eqref{Neumann-bordo-app}. This justifies the study of the equation \eqref{allen-cahn-iper-app} in the space $X$. Indeed, the energy is related to the $X$-norm and, as we will see in the next theorem, the energy dissipation allows us, under certain hypotheses on $f$, to obtain estimates for the solution in $X$ and so global existence for all $\x\in X$.
\begin{thm}\label{global existence}
We consider the equation \eqref{allen-cahn-iper-app} with boundary conditions \eqref{Neumann-bordo-app} and initial data \eqref{cond-iniz-app}. We suppose that $f,g$ are locally Lipschitz on $\mathbb{R}$,
\begin{equation}\label{hp-g}
g(s)\geq0, \qquad \quad \forall\,s\in\mathbb{R}
\end{equation}
and that there is $K>0$ such that for any $|x|>K$
\begin{equation}\label{hp-f}
F(x)\geq Cx^2, \qquad \quad \mbox{ for some }\, C\in\mathbb{R},
\end{equation}
where $\displaystyle F(x):=-\int_0^x f(s)ds$. Then, for any $(u_0,u_1)\in H^1\times L^2$ there exists a unique mild solution 
$$(u,u_t)\in C([0,\infty), H^1\times L^2).$$
\end{thm}
\begin{proof}
Local existence and uniqueness of mild solution follows from Theorem \ref{esist-unic-altern}. We show that $\|(u,u_t)\|_X$ does not tend to infinity as $t\uparrow T(u_0,u_1)$. Let
$$\psi(t)={\|(u(t),u_t(t))\|}^2_X=\int_a^b \left(\varepsilon^2 u^2_x+\tau u^2+\tau u^2_t\right)dx.$$ 
Thanks to the relation \eqref{variazione-energia} and the hypothesis \eqref{hp-g} on $g$ we have
$$E[u,u_t](t)\leq E[u_0,u_1], \qquad \qquad \forall\,t\in[0,T(u_0,u_1)).$$
It follows that
\begin{align*}
\psi(t)& =2E[u,u_t](t)-2\int_a^b F(u(x,t))dx+\int_a^b\tau u^2(x,t)dx  \\
& \leq 2E[u_0,u_1]-2\int_{\left\{|u|\leq K\right\}} F(u)dx-2\int_{\left\{|u|> K\right\}}F(u)dx+\int_a^b \tau u^2(x,t)dx\\ 
& \leq 2E[u_0,u_1]-2\int_{\left\{|u|\leq K\right\}}F(u)dx-2C\int_a^bu^2(x,t)dx \\
&+ 2C\int_{\left\{|u|\leq K\right\}}u^2(x,t)dx+\int_a^b \tau u^2(x,t)dx,
\end{align*}
for all $t\in[0,T(u_0,u_1))$ and so, there exists a constant $C_1\in\mathbb{R}$ such that
\begin{equation}\label{norma-limitata}
\psi(t)\leq C_1+\left(\tau-2C\right)\int_a^b u^2(x,t)dx,
\end{equation}
for all $t\in[0,T(u_0,u_1))$. If $2C\geq\tau$, we obtain $\|(u,u_t)\|_X\leq C_1$ for any $t\geq 0$. Otherwise, we have
\begin{equation}
\int_a^b\tau u^2dx=\int_a^b\tau u_0^2dx+2\int_0^t\int_a^b\tau uu_tdsdx\leq \tau{\|u_0\|}^2_{L^2}+\int_0^t\psi(s)ds. \label{norma-u-L2}
\end{equation}
Substituting \eqref{norma-u-L2} into \eqref{norma-limitata} we obtain
$$\psi(t)\leq C_1+(\tau-2C){\|u_0\|}^2_{L^2}+C_2\int_0^t\psi(s)ds,$$
for all $t\in[0,T(u_0,u_1))$. Applying Gronwall's Lemma, we have
$$\psi(t)\leq C_3e^{C_2t}, \qquad \qquad \forall\, t\in[0,T(u_0,u_1)),$$
and so $T(u_0,u_1)=\infty$.
\end{proof}
\end{appendices}

\vspace{0.5cm}
\textbf{Acknowledgements.} This work was done during my PhD at University of L'Aquila. It was strongly influenced by discussions with C. Lattanzio and C. Mascia. I am very grateful for their invaluable advice. I would also like to thank the referees for their helpful suggestions that have really improved this paper.

\bibliographystyle{plain}
\nocite{*}
\bibliography{slow-motion-hyp-allen-cahn-biblio-arxiv}
\end{document}